\documentclass[12pt]{article}

\usepackage{xcolor}
\usepackage{graphicx}
\usepackage{pictex}
\usepackage{mathrsfs}
\usepackage{amssymb}

\title{A closure for Hamilton-connectedness\\ in $\{\claw,\Gamma_3\}$-free 
       graphs$^{1}$}

\author{Adam Kabela \and Zden\v{e}k Ryj\'a\v{c}ek \and 
        M\'aria Skyvov\'a \and Petr Vr\'ana}
      
\date{July 9, 2024}

\newcounter{mathitem}
\newenvironment{mathitem}
  {\begin{list}{{$(\roman{mathitem})$}}{
   \setcounter{mathitem}{0}
   \usecounter{mathitem}
   \setlength{\topsep}{0pt plus 2pt minus 0pt}
   \setlength{\parskip}{0pt plus 2pt minus 0pt}
   \setlength{\partopsep}{0pt plus 2pt minus 0pt}
   \setlength{\parsep}{0pt plus 2pt minus 0pt}
   \setlength{\leftmargin}{35pt}
   \setlength{\itemsep}{0pt plus 2pt minus 0pt}}}
  {\end{list}}

\newcounter{mylist}
\newenvironment{mylist}
  {\begin{list}{{$\roman{mylist}$}}{
   \setlength{\topsep}{7pt plus 2pt minus 0pt}
   \setlength{\parsep}{3pt plus 2pt minus 0pt}
   \setlength{\leftmargin}{12pt}
   \setlength{\labelwidth}{-6pt}
   }
   }
  {\end{list}}

\newcounter{prostredi}
\def\theprostredi{\arabic{prostredi}}
\def\vejde#1{\unskip
\nobreak\hfill\penalty50\hskip1em\hbox{}\nobreak\hfill
\hbox{#1}}
\newenvironment{theorem}{\par\bigskip\noindent%
\refstepcounter{prostredi}{\bf Theorem \theprostredi.}\quad\bgroup\sl }
{\egroup\par\bigskip\endtrivlist}%
\newenvironment{proposition}{\par\bigskip\noindent%
\refstepcounter{prostredi}{\bf Proposition \theprostredi.}\quad\bgroup\sl }
{\egroup\par\bigskip\endtrivlist}%
\newenvironment{lemma}{\par\bigskip\noindent%
\refstepcounter{prostredi}{\bf Lemma \theprostredi.}\quad\bgroup\sl }
{\egroup\par\bigskip\endtrivlist}%
\newenvironment{proof}{\par
\noindent%
{\bf Proof.}\quad\bgroup}
{\egroup\vejde{\rule{2.5mm}{2.5mm}}\par\bigskip\endtrivlist}%
\newenvironment{corollary}{\par\bigskip\noindent%
\refstepcounter{prostredi}{\bf Corollary
\theprostredi.}\quad\bgroup\sl }
{\egroup\par\bigskip\endtrivlist}%

\newcounter{prostrclaim}
\def\theprostrclaim{\arabic{prostrclaim}}
\def\vejde#1{\unskip
\nobreak\hfill\penalty50\hskip1em\hbox{}\nobreak\hfill
\hbox{#1}}
\newenvironment{claim}{\par\bigskip\noindent%
\refstepcounter{prostrclaim}{\underline{Claim \theprostrclaim.}}\quad\bgroup\sl }
{\egroup\par\bigskip\endtrivlist}%
\newenvironment{proofcl}{\par
\noindent%
{\underline{Proof.}}\quad\bgroup}
{\egroup\vejde{$\Box$}\par\bigskip\endtrivlist}%

\newcounter{prostralph}
\def\theprostralph{\Alph{prostralph}}
\def\vejde#1{\unskip
\nobreak\hfill\penalty50\hskip1em\hbox{}\nobreak\hfill
\hbox{#1}}
\newenvironment{theoremAcite}[1]{\par\bigskip\noindent%
\refstepcounter{prostralph}{\bf Theorem
\theprostralph{} {#1}.}\quad\bgroup\sl }
{\egroup\par\bigskip\endtrivlist}%
\newenvironment{propositionAcite}[1]{\par\bigskip\noindent%
\refstepcounter{prostralph}{\bf Proposition
\theprostralph{} {#1}.}\quad\bgroup\sl }
{\egroup\par\bigskip\endtrivlist}%
\newenvironment{lemmaAcite}[1]{\par\bigskip\noindent%
\refstepcounter{prostralph}{\bf Lemma
\theprostralph{} {#1}.}\quad\bgroup\sl }
{\egroup\par\bigskip\endtrivlist}%

\newcommand{\car}{,\penalty0\relax} 
\newcommand{\indsub}{\stackrel{\mbox{\tiny IND}}{\subset}}
\newcommand{\noi}{\noindent}
\newcommand{\ld}{\ldots}
\newcommand{\iso}{\simeq}
\newcommand{\niso}{\not \simeq}
\newcommand{\sm}{\setminus}
\newcommand{\cl}{{\rm cl}}
\newcommand{\claw}{K_{1,3}}
\newcommand{\Gt}{\Gamma_3}
\newcommand{\Gi}{\Gamma_i}
\newcommand{\la}{\langle}
\newcommand{\lab}{\langle \{}
\newcommand{\ra}{\rangle}
\newcommand{\rag}{\rangle _G}
\newcommand{\ragb}{\rangle _{\bar{G}}}
\newcommand{\rab}{\} \rangle}
\newcommand{\rabg}{\} \rangle _G}
\newcommand{\rabgb}{\} \rangle _{\bar{G}}}
\newcommand{\cF}{{\cal F}}
\newcommand{\cG}{{\cal G}}
\newcommand{\cW}{{\cal W}}
\newcommand{\bG}{\bar{G}}
\newcommand{\Gst}{G^{^*}}
\newcommand{\Gstx}{G^{^*}_x}
\newcommand{\bGst}{\bG^{^*}}
\newcommand{\bGstx}{\bG^{^*}_x}
\newcommand{\NbG}{N_{\bG}}
\newcommand{\VLDC}{V_{\rm{L2C}}}
\newcommand{\bp}{\beginpicture}
\newcommand{\ep}{\endpicture}
\newcommand{\dist}{\mbox{dist}}
\newcommand{\Lp}{L^{-1}}
\newcommand{\bs}{\bigskip}
\newcommand{\bsm}{\vspace{-4mm}}
\newcommand{\ms}{\medskip}
\newcommand{\ssk}{\smallskip}

\begin{document}
\maketitle

\footnotetext[1]{All authors are affiliated with the Department of Mathematics; 
European Centre of Excellence NTIS - New Technologies for the Information 
Society, University of West Bohemia,  Univerzitn\'{\i}~8, 301 00 Pilsen, 
Czech Republic.
E-mails: {\tt $\{$kabela,ryjacek,vranap$\}$@kma.zcu.cz, mskyvova@ntis.zcu.cz.} 
The research was supported by project GA20-09525S of the Czech Science 
Foundation.}

\begin{abstract}
\noi
We introduce a closure technique for Hamilton-connectedness 
of $\{\claw,\Gamma_3\}$-free graphs, where $\Gamma_3$ is the graph obtained 
by joining two vertex-disjoint triangles with a path of length~$3$.
The closure turns a claw-free graph into a line graph of a multigraph while 
preserving its (non)-Hamilton-connectedness. The most technical parts of the 
proof are computer-assisted.

The main application of the closure is given in a subsequent paper
showing that every $3$-connected $\{\claw,\Gamma_3\}$-free graph 
is Hamilton-connected, thus resolving one of the two last open cases in the 
characterization of pairs of connected forbidden subgraphs implying 
Hamilton-connectedness.

\ms

\noi
\end{abstract}

\section{Terminology and notation}
\label{sec-notation}

In this paper, we generally follow the most common graph-theoretical notation 
and terminology, and for notations and concepts not defined here we refer to 
\cite{BM08}.
Specifically, by a {\em graph} we always mean a simple finite undirected graph;
whenever we admit multiple edges, we always speak about a {\em multigraph}.

We write $G_1\subset G_2$ if $G_1$ is a sub(multi)graph of $G_2$, 
$G_1\indsub G_2$ if $G_1$ is an induced sub(multi)graph of $G_2$, 
$G_1\iso G_2$ if the (multi)graphs $G_1$, $G_2$ are isomorphic, 
and $\la M \rag$ to denote the {\em induced sub(multi)graph} on a set 
$M\subset V(G)$.
We use $d_G(x)$ to denote the {\em degree} of a vertex  $x$ in $G$ (note that 
if $G$ is a multigraph, then $d_G(x)$ equals the sum of multiplicities of the
edges containing $x$),
$N_G(x)$ denotes the {\em neighborhood} of a vertex $x$, and $N_G[x]$ the 
{\em closed neighborhood} of $x$, i.e., $N_G[x]=N_G(x)\cup\{x\}$.
For $M\subset V(G)$, we denote $N_M(x)=N_G(x)\cap M$ and 
$N_G[M]=\cup_{x\in M} N_G[x]$.
For $x,y\in V(G)$, $\dist_G(x,y)$ denotes their {\em distance}, i.e., the 
length of a shortest $(x,y)$-path in $G$. More generally, if $F\subset G$ 
is connected and $x,y\in V(F)$, then $\dist_F(x,y)$ denotes 
the length of a shortest $(x,y)$-path in $F$.
By a {\em clique} in $G$ we mean a complete subgraph of $G$ (not necessarily 
maximal), and $\alpha(G)$ denotes the {\em independence number} of $G$.

We say that a vertex $x\in V(G)$ is {\em simplicial} if $\la N_G(x)\rag$ is a 
clique, and we use $V_{SI}(G)$ to denote the set of all simplicial 
vertices of $G$, and $V_{NS}(G)=V(G)\setminus V_{SI}(G)$ the set of 
nonsimplicial vertices of $G$.
For $k\geq 1$, we say that a vertex $x\in V(G)$ is {\em locally $k$-connected 
in $G$} if $\la N_G(x)\rag$ is a $k$-connected graph.

A graph is {\em Hamilton-connected} if, for any $u,v\in V(G)$, $G$ has a 
hamiltonian $(u,v)$-path, i.e., an $(u,v)$-path $P$ with $V(P)=V(G)$.

Finally, if $\cF$ is a family of graphs, we say that $G$ is {\em $\cF$-free}
if $G$ does not contain an induced subgraph isomorphic to a member of $\cF$,
and the graphs in $\cF$ are referred to in this context as {\em forbidden
(induced) subgraphs}.
If $\cF=\{F\}$, we simply say that $G$ is {\em $F$-free}.
Here, the {\em claw} is the graph $\claw$, $P_i$ denotes the path on $i$ 
vertices, and $\Gamma_i$ denotes the graph obtained by joining two triangles 
with a path of length $i$ (see Fig.~\ref{fig-special_graphs}$(d)$).
Several further graphs that will occur as forbidden subgraphs are 
shown in Fig.~\ref{fig-special_graphs}$(a),(b),(c)$.
Whenever we will list vertices of an induced claw $\claw$, we will always list 
its center as the first vertex of the list, and when listing vertices of an 
induced subgraph $\Gamma_i$, we always list first the vertices of degree 2 
of one of the triangles, then the vertices of the path, and we finish with 
the vertices of degree 2 of the second triangle (i.e., in the labeling of
vertices as in Fig.~\ref{fig-special_graphs}$(d)$, we write 
$\lab t_1,t_2,p_1,\ldots,p_{i+1},t_3,t_4\rabg\iso\Gi$).

\bsm

%
%
\begin{figure}[ht]
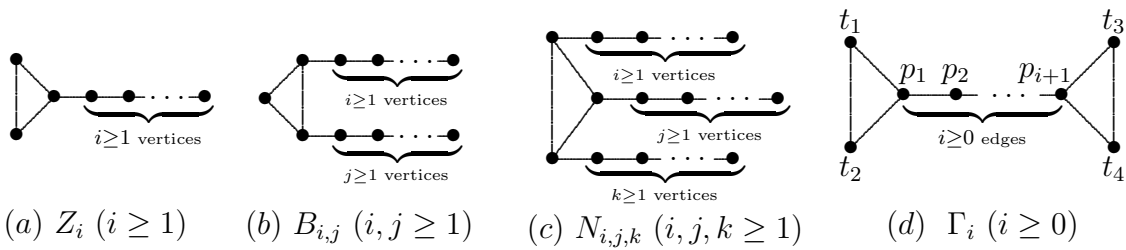

$$\bp
\setcoordinatesystem units <0.75mm,1.0mm>
\setplotarea x from -50 to 50, y from -5 to 5
\put
 {\bp
 \setcoordinatesystem units <1.0mm,1.0mm>
 \setplotarea x from 0 to 25, y from -5 to 5
 \put{$\bullet$} at    0    5
 \put{$\bullet$} at    0   -5
 \put{$\bullet$} at    5    0
 \put{$\bullet$} at   10    0
 \put{$\bullet$} at   15    0
 \put{$\ldots$}  at   20    0
 \put{$\bullet$} at   25    0
 \plot 5 0  0 -5  0 5  5 0  17 0 /
 \plot 23 0  25 0 /
 \put{$\underbrace{\hspace*{1.7cm}}_{i\geq 1\mbox{\tiny\ vertices}}$}
                             at 17.5 -4
 \put{$(a) \ Z_i$ $(i\geq 1$)} at  10.5 -17
 \ep}
at  -85   -1
\put
 {\bp
 \setcoordinatesystem units <1.0mm,1.0mm>
 \setplotarea x from -5 to 25, y from -5 to 5
 \put{$\bullet$} at    0    0
 \put{$\bullet$} at    5    5
 \put{$\bullet$} at   10    5
 \put{$\bullet$} at   15    5
 \put{$\ldots$}  at   20    5
 \put{$\bullet$} at   25    5
 \put{$\bullet$} at    5   -5
 \put{$\bullet$} at   10   -5
 \put{$\bullet$} at   15   -5
 \put{$\ldots$}  at   20   -5
 \put{$\bullet$} at   25   -5
 \plot 17 5 5 5 0 0 5 -5 17 -5 /
 \plot 23  5  25  5 /
 \plot 23 -5  25 -5 /
 \plot 5 -5  5 5 /
 \put{$\underbrace{\hspace*{1.7cm}}_{^{i\geq 1\mbox{\tiny\ vertices}}}$}
                             at 17.5  1.0
 \put{$\underbrace{\hspace*{1.7cm}}_{^{j\geq 1\mbox{\tiny\ vertices}}}$}
                             at 17.5  -9.0
 \put{$(b) \ B_{i,j}$ ($i,j\geq 1$)} at  12.5 -17
 \ep}
at   -42   -1.2
\put
 {\bp
 \setcoordinatesystem units <1.2mm,0.8mm>
 \setplotarea x from -25 to 25, y from -5 to 5
 \put{$\bullet$} at    0    5
 \put{$\bullet$} at    5    5
 \put{$\bullet$} at   10    5
 \put{$\ldots$}  at   15    5
 \put{$\bullet$} at   20    5
 \put{$\bullet$} at    5   -5
 \put{$\bullet$} at   10   -5
 \put{$\bullet$} at   15   -5
 \put{$\ldots$}  at   20   -5
 \put{$\bullet$} at   25   -5
 \put{$\bullet$} at    0  -15
 \put{$\bullet$} at    5  -15
 \put{$\bullet$} at   10  -15
 \put{$\ldots$}  at   15  -15
 \put{$\bullet$} at   20  -15
 \plot 0 5 5 -5 0 -15 0 5 /
 \plot 0  5  12 5 /
 \plot  18 5  20 5 /
 \plot 5 -5  17 -5 /
 \plot  25 -5  23 -5 /
 \plot 0  -15  12 -15 /
 \plot  18 -15  20 -15 /
 \put{$\underbrace{\hspace*{2.1cm}}_{^{i\geq 1\mbox{\tiny\ vertices}}}$}
                             at 12.5  0.4
 \put{$\underbrace{\hspace*{2.1cm}}_{^{j\geq 1\mbox{\tiny\ vertices}}}$}
                             at 17.5  -9.6
 \put{$\underbrace{\hspace*{2.1cm}}_{^{k\geq 1\mbox{\tiny\ vertices}}}$}
                             at 12.5  -19.6
 \put{$(c) \ N_{i,j,k}$ ($i,j,k\geq 1$)} at  12.5 -27
 \ep}
at   -4  0
\put
 {\bp
 \setcoordinatesystem units <1.4mm,1.4mm>
 \setplotarea x from 0 to 25, y from -5 to 5
 \put{$\bullet$} at    0   5
 \put{$\bullet$} at    0  -5
 \put{$\bullet$} at    5   0
 \put{$\bullet$} at   10   0
 \put{$\ldots$}  at   15   0
 \put{$\bullet$} at   20   0
 \put{$\bullet$} at   25   5
 \put{$\bullet$} at   25  -5
 \plot 5 0  0 -5  0 5  5 0  12 0 /
 \plot 18 0  20 0  25 5  25 -5  20 0 /
 \put{$t_1$}   at   0   7
 \put{$t_2$}   at   0  -7
 \put{$p_1$}   at   6  2
 \put{$p_2$}   at  10  2
 \put{$p_{i+1}$}  at  18.5  2
 \put{$t_3$}   at  25   7
 \put{$t_4$}   at  25  -7
  \put{$\underbrace{\hspace*{2.1cm}}_{i\geq 0\mbox{\tiny\ edges}}$}
                             at 12.5 -3.1
 \put{$(d) \ \ \Gamma_i$ $(i\geq 0)$} at  12.5 -12.5
 \ep}
at  70 1.8
\ep$$
\vspace*{-5mm}
\caption{The graphs $Z_i$, $B_{i,j}$, $N_{i,j,k}$ and $\Gi$}
\label{fig-special_graphs}
\end{figure}

\section{Introduction}
\label{sec-intro}

There are many results on forbidden induced subgraphs implying various 
Hamilton-type properties. While forbidden pairs of connected graphs 
for hamiltonicity in 2-connected graphs were completely characterized already 
in the early 90's \cite{Be91,FG97}, the progress in forbidden pairs for 
Hamilton-connectedness is relatively slow. 

\ms

Let $W$ denote the Wagner graph and $W^+$ the graph obtained from $W$ by
attaching exactly one pendant edge to each of its vertices 
(see Fig.~\ref{fig-Wagner}).

%
%
\begin{figure}[ht]
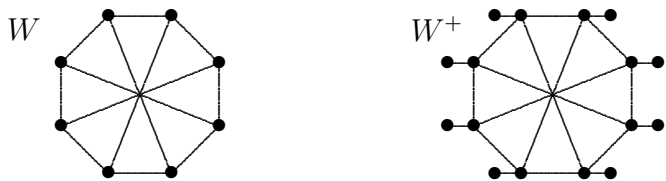

$$\beginpicture
\setcoordinatesystem units <0.7mm,1mm>
\setplotarea x from -70 to 70, y from -5 to 5
\put {\bp
\setcoordinatesystem units <0.35mm,0.35mm>
  \put{$\bullet$} at   -30   12
  \put{$\bullet$} at   -30  -12
  \put{$\bullet$} at    30   12
  \put{$\bullet$} at    30  -12
  \put{$\bullet$} at   -12   30
  \put{$\bullet$} at    12   30
  \put{$\bullet$} at   -12  -30
  \put{$\bullet$} at    12  -30
  \plot -30 -12 -30 12 -12 30 12 30 30 12 30 -12 12 -30
        -12 -30 -30 -12 /
  \plot -30  12  30 -12 /
  \plot -30 -12  30  12 /
  \plot -12  30  12 -30 /
  \plot -12 -30  12  30 /
 \put{$W$}   at  -44    25
\ep} at -40 0
\put {\bp
\setcoordinatesystem units <0.35mm,0.35mm>
  \put{$\bullet$} at   -30   12
  \put{$\bullet$} at   -30  -12
  \put{$\bullet$} at    30   12
  \put{$\bullet$} at    30  -12
  \put{$\bullet$} at   -12   30
  \put{$\bullet$} at    12   30
  \put{$\bullet$} at   -12  -30
  \put{$\bullet$} at    12  -30
  \plot -30 -12 -30 12 -12 30 12 30 30 12 30 -12 12 -30
        -12 -30 -30 -12 /
  \plot -30  12  30 -12 /
  \plot -30 -12  30  12 /
  \plot -12  30  12 -30 /
  \plot -12 -30  12  30 /
  \put{$\bullet$} at   -40   12
  \put{$\bullet$} at   -40  -12
  \put{$\bullet$} at    40   12
  \put{$\bullet$} at    40  -12
  \put{$\bullet$} at   -22   30
  \put{$\bullet$} at    22   30
  \put{$\bullet$} at   -22  -30
  \put{$\bullet$} at    22  -30
  \plot -40  12 -30  12 /
  \plot -40 -12 -30 -12 /
  \plot  40  12  30  12 /
  \plot  40 -12  30 -12 /
  \plot -22  30 -12  30 /
  \plot -22 -30 -12 -30 /
  \plot  22  30  12  30 /
  \plot  22 -30  12 -30 /
 \put{$W^+$}   at  -44    25
\ep} at 40 0
\endpicture$$
\vspace*{-6mm}
\caption{The Wagner graph $W$ and the graph $W^+$}
\label{fig-Wagner}
\vspace*{-2mm}
\end{figure}

Theorem~\ref{thmA-known_results} below lists the best known results on pairs of 
forbidden subgraphs implying Hamilton-connectedness of a 3-connected graph.

%
%
\begin{theoremAcite}{\cite{BGHJFW14,BFHTV02,LRVXY23-I,LRVXY23-II,LXL21,RV21,RV23}}
\label{thmA-known_results}
Let $G$ be a 3-connected $\{\claw,X\}$-free graph, where
\begin{mathitem}
\item {\bf ~\cite{BFHTV02}} $X=\Gamma_1$, or
\item {\bf ~\cite{BGHJFW14}} $X=P_9$, or
\item {\bf ~\cite{RV21}} $X=Z_7$ and $G\niso L(W^+)$, or 
\item {\bf ~\cite{RV23}} $X=B_{i,j}$ for $i+j\leq 7$, or
\item {\bf ~\cite{LRVXY23-I,LRVXY23-II,LXL21}} $X=N_{i,j,k}$ for $i+j+k\leq 7$.
\end{mathitem} 
Then $G$ is Hamilton-connected.
\end{theoremAcite}

Let $\cW$ be the family of graphs obtained by attaching at least one pendant edge
to each of the vertices of the Wagner graph $W$, and let 
$\cG=\{L(H)|\ H\in\cW\}$ be the family of their line graphs.
Then any $G\in\cG$ is 3-connected, non-Hamilton-connected, $P_{10}$-free,
$Z_8$-free, $B_{i,j}$-free for $i+j=8$ and $N_{i,j,k}$-free for $i+j+k=8$. 
Thus, this example shows that parts $(ii)$, $(iii)$, $(iv)$ and $(v)$ of 
Theorem~\ref{thmA-known_results} are sharp.

\bs

According to the discussion in Section~6 of \cite{RV23}, there are two
remaining connected graphs $X$ that might imply Hamilton-connectedness 
of a 3-connected $\{\claw,X\}$-free graph, namely, the graph $\Gt$, 
and the graph $\Gamma_5$ for $|V(G)|\geq 21$ (or possibly with a single 
exception of $G\iso L(W^+)$). In this paper, we address the first of these
graphs, the graph $\Gt$. We develop the main tool, the 
closure operation, and in the subsequent paper \cite{KRSV???-II}, as an
application of the main result of this paper, we prove the following.

%
%
\begin{theoremAcite}{\cite{KRSV???-II}}
\label{thmA-main}
Every $3$-connected $\{\claw,\Gamma_3\}$-free graph is Hamilton-connected.
\end{theoremAcite}

In Section~\ref{sec-preliminaries}, we collect necessary known results and 
facts on line graphs and on closure operations, and then, in 
Section~\ref{sec-Gt-closure}, we develop a closure technique that will be
crucial for the proof of Theorem~\ref{thmA-main}.
The most technical parts of the proofs (namely, Case~1 and Subcase~2.2 in the
proof of Proposition~\ref{prop-closure-W4}) are computer-assisted.
More details on the computation can be found in Section~\ref{sec-concluding}, 
and detailed results of the computation and source codes are available in 
the repository at link~\cite{computing}.

\section{Preliminaries}
\label{sec-preliminaries}

In this section, we summarize known facts that will be needed in the 
proof of our main result, Theorem~\ref{thm-Gamma3-clos-lajngraf}.

\subsection{Line graphs of multigraphs and their preimages}
\label{subsec-linegraph-preimage}

The following characterization of line graphs of multigraphs was proved 
by Bermond and Meyer~\cite{BM73} (see also Zverovich~\cite{Z97}).

%
%
\begin{theoremAcite}{\cite{BM73}}
\label{thmA-BeMe}
A graph $G$ is a line graph {\em of a multigraph} if and only if
$G$ does not contain a copy of any of the graphs in Figure~\ref{fig-BeMe} 
as an induced subgraph.
\end{theoremAcite}

\vspace*{-5mm}

%
%
\begin{figure}[ht]
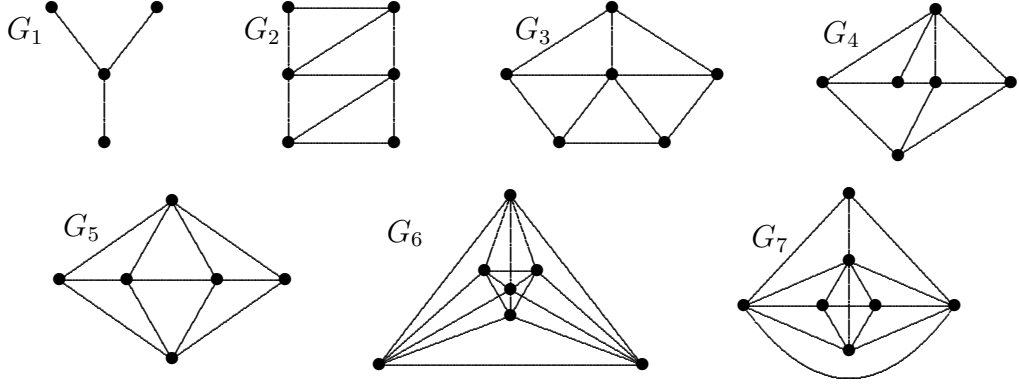

$$\bp
\setcoordinatesystem units <0.9mm,0.75mm>
\setplotarea x from -50 to 50, y from 0 to 45
\put{
\bp
\setcoordinatesystem units <0.7mm,.6mm>
\setplotarea x from -20 to 20, y from -20 to 15
\put{$\bullet$} at -10  15
\put{$\bullet$} at  10  15
\put{$\bullet$} at   0   0
\put{$\bullet$} at   0 -15
\plot -10 15  0 0  10 15 /
\plot 0 0  0 -15 /
\put{$G_1$} at -15 10
\ep} at -60 35
\put{
\bp
\setcoordinatesystem units <0.7mm,.6mm>
\setplotarea x from -20 to 20, y from -20 to 15
\put{$\bullet$} at -10  15
\put{$\bullet$} at  10  15
\put{$\bullet$} at -10   0
\put{$\bullet$} at  10   0
\put{$\bullet$} at -10 -15
\put{$\bullet$} at  10 -15
\plot -10 -15 -10 15 10 15 10 -15 -10 -15 10 0 -10 0 10 15 /
\put{$G_2$} at -15 10
\ep} at -25 35
\put{
\bp
\setcoordinatesystem units <0.7mm,.6mm>
\setplotarea x from -20 to 20, y from -20 to 15
\put{$\bullet$} at   0  15
\put{$\bullet$} at   0   0
\put{$\bullet$} at -20   0
\put{$\bullet$} at  20   0
\put{$\bullet$} at -10 -15
\put{$\bullet$} at  10 -15
\plot 0 15 -20 0 -10 -15 10 -15 20 0 0 15 0 0 -10 -15 /
\plot -20 0 20 0 /
\plot 0 0  10 -15 /
\put{$G_3$} at -15 10
\ep} at 15 35
\put{
\bp
\setcoordinatesystem units <0.5mm,.65mm>
\setplotarea x from -20 to 20, y from -20 to 20
\put{$\bullet$} at -25   0
\put{$\bullet$} at  -5   0
\put{$\bullet$} at   5   0
\put{$\bullet$} at  25   0
\put{$\bullet$} at   5  15
\put{$\bullet$} at  -5 -15
\plot -25 0 25 0  5 15 -25 0  -5 -15 5 0  5 15  -5 0 /
\plot -5 -15  25 0 /
\put{$G_4$} at -20 10
\ep} at  60 35
\put{
\bp
\setcoordinatesystem units <0.6mm,.7mm>
\setplotarea x from -20 to 20, y from -20 to 20
\put{$\bullet$} at   0  15
\put{$\bullet$} at   0 -15
\put{$\bullet$} at -25   0
\put{$\bullet$} at -10   0
\put{$\bullet$} at  10   0
\put{$\bullet$} at  25   0
\plot -25 0  25 0  0 15 -25 0  0 -15 -10 0  0 15  10 0  0 -15 25 0 /
\put{$G_5$} at -20 10
\ep} at -50 0
\put{
\bp
\setcoordinatesystem units <0.7mm,.5mm>
\setplotarea x from -20 to 20, y from -20 to 20
\put{$\bullet$} at   5  20
\put{$\bullet$} at   0   0
\put{$\bullet$} at  10   0
\put{$\bullet$} at   5  -5
\put{$\bullet$} at   5 -12
\put{$\bullet$} at -20 -25
\put{$\bullet$} at  30 -25
\plot 5 20 -20 -25  30 -25  5 20  0 0  -20 -25  5 -12  30 -25  10 0
    5 20  5 -12  0 0  10 0  5 -12 /
\plot -20 -25  5 -5  0 0 /
\plot 30 -25  5 -5  10 0 /
\put{$G_6$} at -15 10
\ep} at   0 0
\put{
\bp
\setcoordinatesystem units <0.7mm,.6mm>
\setplotarea x from -20 to 20, y from -20 to 20
\put{$\bullet$} at   0  20
\put{$\bullet$} at   0   5
\put{$\bullet$} at -20  -5
\put{$\bullet$} at  20  -5
\put{$\bullet$} at  -5  -5
\put{$\bullet$} at   5  -5
\put{$\bullet$} at  0  -15
\plot 0 20 -20 -5  20 -5  0 20  0 -15 -20 -5  0 5  20 -5
   0 -15  -5 -5  0 5  5 -5  0 -15 /
\setquadratic
\plot -20 -5  0 -21  20 -5 /
\setlinear
\put{$G_7$} at -15 10
\ep} at   50 0
\ep$$
\caption{Forbidden subgraphs for line graphs of multigraphs}
\label{fig-BeMe}
\end{figure}

While in line graphs of graphs, for a connected line graph $G$, the graph $H$ 
such that $G=L(H)$ is uniquely determined with a single exception of $G=K_3$, 
in line graphs of multigraphs this is not true: a simple example are the graphs 
$H_1=Z_1$ and $H_2$ a double edge with one pendant edge attached to each vertex
-- while $H_1\not\iso H_2$, we have $L(H_1)\iso L(H_2)$.
Using a modification of an approach from \cite{Z97}, the following was proved 
in \cite{RV11}.

%
%
\begin{theoremAcite}{\cite{RV11}}
\label{thmA-vzor_jedn}
Let $G$ be a connected line graph of a multigraph. Then there is, up
to an isomorphism, a uniquely determined multigraph $H$ such that 
a vertex $e\in V(G)$ is simplicial in $G$ if and only if the corresponding 
edge $e\in E(H)$ is a pendant edge in $H$.
\end{theoremAcite}

The multigraph $H$ with the properties given in Theorem~\ref{thmA-vzor_jedn}
will be called the {\em preimage} of a line graph $G$ and denoted $H=\Lp(G)$.
We will also use the notation $a=L(e)$ and $e=\Lp(a)$ for an edge $e\in E(H)$
and the corresponding vertex $a\in V(G)$.

\ms

An edge-cut $R\subset E(H)$ of a multigraph $H$ is {\em essential} if $H-R$ has 
at least two nontrivial components, and $H$ is 
{\em essentially $k$-edge-connected} if every essential edge-cut of $H$ is 
of size at least $k$. It is obvious that 
a set $M\subset V(G)$ of vertices of a line graph $G$ is a vertex-cut of $G$ 
if and only if the corresponding set $\Lp(M)\subset E(\Lp(G))$ is an essential
edge-cut of $\Lp(G)$. Consequently, a noncomplete line graph $G$ is 
$k$-connected if and only if $\Lp(G)$ is essentially $k$-edge-connected.
It is also a well-known fact that if $X$ is a line graph, then a line graph $G$
is $X$-free if and only if $\Lp(G)$ does not contain as a subgraph (not 
necessarily induced) a graph $F$ such that $L(F)=X$.

\subsection{Closure operations}
\label{subsec-closure}

For $x\in V(G)$, the {\em local completion of $G$ at $x$} is the
graph $G^{^*}_x=(V(G),E(G)\cup\{y_1y_2|\ y_1,y_2\in N_G(x)\})$ (i.e.,
$G^{^*}_x$ is obtained from $G$ by adding all the missing edges with
both vertices in $N_G(x)$). In this context, the edges in $E(\Gstx)\sm E(G)$
will be refereed to as {\em new edges}, and the edges in $E(G)$ are {\em old}.
Obviously, if $G$ is claw-free, then so is $G^{^*}_x$.
Note that in the special case when $G$ is a line graph and $H=\Lp(G)$, 
$G^{^*}_x$ is the line graph of the multigraph $H|_{e}$ obtained from $H$ by 
contracting the edge $e=\Lp(x)$ into a vertex and replacing the created 
loop(s) by pendant edge(s) (Thus, if $G=L(H)$ and $x=L(e)$, then 
$\Gst_x=L(H|_e)$).

Also note that clearly $x\in V_{SI}(G^{^*}_x)$ for any $x\in V(G)$, and, more 
generally, $V_{SI}(G)\subset V_{SI}(G^{^*}_x)$ for any $x\in V(G)$.

\ms

We say that a vertex $x\in V(G)$ is {\em eligible} if $\la N_G(x)\rag$ 
is a connected noncomplete graph, and we use $V_{EL}(G)$ to denote the set of 
all eligible vertices of $G$. Note that in the special case when $G$ is a 
line graph and $H=\Lp(G)$, it is not difficult to observe that $x\in V(G)$ is 
eligible if and only if the edge $\Lp(x)$ is in a triangle or in a multiple edge
of $H$.
Based on the fact that if $G$ is claw-free and $x\in V_{EL}(G)$, then
$G^{^*}_x$ is hamiltonian if and only if $G$ is hamiltonian, the
{\em closure} $\cl(G)$ of a claw-free graph $G$ was defined in \cite{R97} as
the graph obtained from $G$ by recursively performing the local completion
operation at eligible vertices, as long as this is possible (more precisely:
$\cl(G)=G_k$, where $G_1,\ld,G_k$ is a sequence of graphs such that $G_1=G$,
$G_{i+1}=(G_i)^{^*}_{x_i}$ for some $x_i\in V_{EL}(G)$, $i=1,\ld,k-1$,
and $V_{EL}(G_k)=\emptyset$).
The closure $\cl(G)$ of a claw-free graph $G$ is uniquely determined, is a 
line graph of a triangle-free graph, and is hamiltonian if and only if so 
is $G$. However, as observed in \cite{BFR00}, the closure operation
does not preserve the (non-)Hamilton-connectedness of~$G$.

\bs

In attempts to identify reasons of this problem, the following result 
on stability of hamiltonian path under local completion was 
proved in \cite{BFR00}.

%
%
\begin{propositionAcite}{\cite{BFR00}}
\label{propA-subgraph_S}
Let $x$ be an eligible vertex of a claw-free graph $G$, $\Gstx$ the
local completion of $G$ at $x$, and $a$, $b$ two distinct vertices
of $G$. Then for every longest $(a,b)$-path $P'(a,b)$ in $\Gstx$
there is a path $P$ in $G$ such that $V(P)=V(P')$ and $P$ admits at
least one of $a$, $b$ as an endvertex. Moreover, there is an
$(a,b)$-path $P(a,b)$ in $G$ such that $V(P)=V(P')$ except perhaps
in each of the following two situations (up to symmetry between $a$
and $b$):
\begin{mathitem}
\item There is an induced subgraph $H \subset G$ isomorphic to the
     graph $S$ in Figure~\ref{fig-graph_S} such that both $a$ and $x$ are 
     vertices of degree 4 in $H$. In this case $G$ contains a path $P_b$ 
     such that $b$ is an endvertex of $P$ and $V(P_b)=V(P')$. If, moreover,
     $b\in V(H)$, then $G$ contains also a path $P_a$ with endvertex
     $a$ and with $V(P_a)=V(P')$.
\item $x=a$ and $ab\in E(G)$. In this case there is always both a path
     $P_a$ in $G$ with endvertex $a$ and with $V(P_a)=V(P')$ and
     a path $P_b$ in $G$ with endvertex $b$ and with $V(P_b)=V(P')$.
\end{mathitem}
\end{propositionAcite}

\bsm

%
%
\begin{figure}[ht]
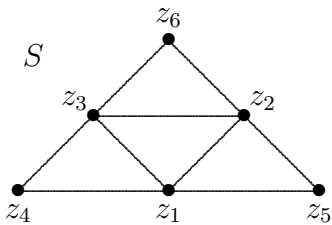

$$\beginpicture
\setcoordinatesystem units <1mm,1mm>
\setplotarea x from -30 to 30, y from 10 to 20
\put {$\bullet$} at    0  10
\put {$z_1$}     at    0   7
\put {$\bullet$} at    0  30
\put {$z_6$}     at    0  33
\put {$\bullet$} at  -20  10
\put {$z_4$}     at  -20   7
\put {$\bullet$} at  -10  20
\put {$z_3$}     at -12.5 22
\put {$\bullet$} at   20  10
\put {$z_5$}     at   20   7
\put {$\bullet$} at   10  20
\put {$z_2$}     at 12.5  22
\setlinear
\plot -20 10  20 10  0 30  -10 20  0 10  10 20  -10 20  -20 10 /
\put {$S$} at  -18  28
\endpicture$$
\vspace{-10mm}
\caption{The graph $S$}
\label{fig-graph_S}
\end{figure}

The following consequence of Proposition~\ref{propA-subgraph_S} will be
useful in our proof.

%
%
\begin{corollary}
\label{coro-subgraph_S}
Let $G$ be a claw-free graph that is not Hamilton-connected, and let
$x_1,x_2\in V_{EL}(G)$ be such that $x_1x_2\notin E(G)$ and both 
$\Gst_{x_1}$ and $\Gst_{x_2}$ are Hamilton-connected.
Then at least one of the vertices $x_1,x_2$ is a vertex of degree 4 
in an induced subgraph $H\indsub G$ isomorphic to the graph $S$ in 
Figure~\ref{fig-graph_S}.
\end{corollary}

\begin{proof}
If at least one of $x_1,x_2$ satisfies $(i)$ of 
Proposition~\ref{propA-subgraph_S}, then we are done; thus, let both $x_1$
and $x_2$ satisfy $(ii)$ of Proposition~\ref{propA-subgraph_S}.
Thus, by $(ii)$, there are vertices $b_1,b_2\in V(G)$ such that 
$x_ib_i\in E(G)$ and $G$ has no hamiltonian $(x_i,b_i)$-path, $i=1,2$.
Since $x_1x_2\notin E(G)$, $b_1\neq x_2$ and $b_2\neq x_1$. 
Thus, again by $(ii)$ of Proposition~\ref{propA-subgraph_S}, there is no
hamiltonian $(x_2,b_2)$-path in $\Gst_{x_1}$, implying that $\Gst_{x_1}$
is not Hamilton-connected, a contradiction.
\end{proof}

To handle the problem of unstability of Hamilton-connectedness, the closure 
concept was strengthened in \cite{KRTV12} and \cite{LRVXY23-I} 
such that the closure operation (called SM-closure in \cite{KRTV12}, and
its further strengthening, UM-closure in \cite{LRVXY23-I}), preserves 
the (non)-Hamilton-connectedness. However, these operations are
not applicable to $\{\claw,\Gt\}$-free graphs since a closure of a 
$\{\claw,\Gt\}$-free graph is not necessarily $\Gt$-free.

Before showing a way to overcome this problem, we first recall two 
classical results by Chv\'atal and Erd\H{o}s~\cite{ChE72}
and by Fouquet~\cite{F93} that will be needed in the proof of the 
main result.

%
%
\begin{theoremAcite}{\cite{ChE72}}
\label{thmA-Chvatal-Erdos}
Let $G$ be an $s$-connected graph containing no independent set of 
$s$ vertices. Then $G$ is Hamilton-connected. 
\end{theoremAcite}

%
%
\begin{theoremAcite}{\cite{F93}}
\label{thmA-Fouquet}
Let $G$ be a connected claw-free graph with independence number at least 
three. Then every vertex $v$ satisfies exactly one of the following:
\begin{mathitem}
\item $N(v)$ is covered by two cliques,
\item $\la N(v)\rag$ contains an induced $C_5$.
\end{mathitem}
\end{theoremAcite}

\section{$\Gt$-closure}
\label{sec-Gt-closure}

As already mentioned, the SM-closure and UM-closure operations 
(see \cite{KRTV12} and \cite{LRVXY23-I}) preserve 
Hamilton-connectedness, but there is still a problem that the local 
completion $G^{^*}_{x}$ of a $\{\claw,\Gt\}$-free graph $G$ is not 
necessarily $\Gt$-free. 
To handle this problem, we define the concept of a $\Gt$-closure $G^{\Gt}$ 
of a $\{\claw,\Gt\}$-free graph $G$.
For a set $M=\{x_1,x_2,\ldots,x_k\}\subset V(G)$, we set 
$G^{^*}_M=((G^{^*}_{x_1})^{^*}_{x_2}\ldots )^{^*}_{x_k}$. It is 
implicit in the proof of uniqueness of $\cl(G)$ in \cite{R97} (and easy
to see) that, for a given set $M=\{x_1,x_2,\ldots,x_k\}\subset V(G)$,
$G^{^*}_M$ is uniquely determined (i.e., does not depend on the order of the 
vertices $x_1,x_2,\ldots,x_k$ used during the construction).

\ms

If $G$ is not Hamilton-connected, then a vertex $x\in V_{NS}(G)$, for 
which the graph $\Gstx$ is still not Hamilton-connected, is said to be 
{\em feasible in $G$}. A set of vertices $M\subset V(G)$ is said to be 
{\em feasible in $G$} if the vertices in $M$ can be ordered in a sequence 
$x_1,\ldots,x_k$ such that $x_1$ is feasible in $G_0=G$, and $x_{i+1}$
is feasible in $G_i=(G_{i-1})^{^*}_{x_i}$, $i=1,\ldots,k-1$.
Thus, if $M\subset V(G)$ is feasible, then $M\subset V_{SI}(\Gst_M)$, 
but $\Gst_M$ is still not Hamilton-conected.

Note that it is possible that some two vertices $x,y$ of a graph $G$ 
are feasible in $G$, but $x$ is not feasible in $\Gst_y$ 
(for example, if $H$ is obtained from the Petersen graph by adding a
pendant edge to each vertex, subdividing a nonpendant edge $x_1x_2$ 
with a vertex $w$, replacing each of the edges $x_iw$ with a double edge, 
and if $G=L(H)$ and $x'_i,x''_i\in V(G)$ correspond to the two edges joining 
$x_i$ and $w$ in $H$, $i=1,2$, then $G$ is not Hamilton-connected, each of
the vertices $x'_i,x''_i$ is feasible in $G$, $i=1,2$, but e.g. $x'_1$
and $x''_1$ are not feasible in $\Gst_{x'_2}\iso\Gst_{x''_2}$).
Thus, the recursive form of the definition is essential for verifying 
feasibility of a set $M\subset V(G)$ (although the resulting graph $\Gst_M$ 
does not depend on their order).

\ms

Now, for a $\{\claw,\Gt\}$-free graph $G$, we define its $\Gt$-closure 
$G^{\Gt}$ by the following construction.
\begin{mathitem}
\item If $G$ is Hamilton-connected, we define $G^{\Gt}$ as the complete 
  graph.
\item If $G$ is not Hamilton-connected, we recursively perform the
  local completion operation at such feasible sets of vertices for which 
  the resulting graph is still $\Gt$-free, as long as this is possible. 
  We obtain a sequence of graphs $G_1,\ld,G_k$ such that
  \begin{mathitem}
    \item[$\bullet$] $G_1=G$,
    \item[$\bullet$] $G_{i+1}=(G_i)^{^*}_{M_i}$ for some
        set $M_i\subset V(G_i)$, $i=1,\ld,k-1$,
    \item[$\bullet$] $G_k$ has no hamiltonian $(a,b)$-path for some
        $a,b\in V(G_k)$,
    \item[$\bullet$] for any feasible set $M\subset V_{NS}(G_k)$, 
    $(G_k)^{^*}_M$ contains an induced subgraph isomorphic to $\Gt$,
  \end{mathitem}
  and we set $G^{\Gt}=G_k$.
\end{mathitem}
A resulting graph $G^{\Gt}$ is called a {\em $\Gt$-closure} of the graph 
$G$, and a graph $G$ equal to (some) its $\Gt$-closure is said to be 
{\em $\Gt$-closed}.
Note that for a given graph $G$, its $\Gt$-closure is not uniquely determined.

\ms 

The following two results from \cite{RV10} and \cite{RV11} will be useful 
to identify feasible vertices.

%
%
\begin{theoremAcite}{\cite{RV10}}
\label{thmA-HC-2-conn_neighb}
Let $G$ be a claw-free graph and let $x\in V(G)$ be locally 2-connected
in $G$. Then $G$ is Hamilton-connected if and only if $\Gstx$ is 
Hamilton-connected.
\end{theoremAcite}

Thus, in our terminology, Theorem~\ref{thmA-HC-2-conn_neighb} says that 
a locally 2-connected vertex is feasible.

%
%
\begin{lemmaAcite}{\cite{RV11}}
\label{lemmaA-2_indep_sets}
Let $G$ be a claw-free graph, $x\in V(G)$, and let
$H\indsub \la N_G(x)\rag$ be a 2-connected graph containing two
disjoint pairs of independent vertices. Then $x$ is locally 2-connected in $G$.
\end{lemmaAcite}

Note that if a vertex $x\in V(G)$ is feasible by virtue of 
Theorem~\ref{thmA-HC-2-conn_neighb} (i.e., $x$ is locally 2-connected in $G$), 
then, for any $y\in V(G)$ $y\neq x$, $x$ is locally 2-connected also in $\Gst_y$, 
but $\la N_{\Gst_y}(x)\ra_{\Gst_y}$ can be complete (if $N_G(x)\subset N_G(y)$).
Thus, for any $y\in V(G)$, $x$ is feasible or simplicial in $\Gst_y$. 

We thus define more generally: a set $M\subset V(G)$ is {\em weakly feasible 
in $G$} if the vertices in $M$ can be ordered in a sequence $x_1,\ldots,x_k$ 
such that $x_1$ is feasible in $G_0=G$, and $x_{i+1}$ is feasible or simplicial 
in $G_i=(G_{i-1})^{^*}_{x_i}$, $i=1,\ldots,k-1$.
Thus, similarly, if $G$ is not Hamilton-connected and $M\subset V(G)$ is weakly 
feasible in $G$, then $\Gst_M$ is still not Hamilton-conected and all vertices 
of $M$ are simplicial in $\Gst_M$.

\bs

The following theorem is the main result of this paper.

%
%
\begin{theorem}
\label{thm-Gamma3-clos-lajngraf}
Let $G$ be a 3-connected $\{\claw,\Gt\}$-free graph and let $G^{\Gt}$ be its 
$\Gt$-closure. Then there is a multigraph $H$ such that $G^{\Gt}=L(H)$.
\end{theorem}

\noi
{\bf Proof.} \quad 
Let $G$ be a $\{\claw,\Gt\}$-free graph, and let $\bG$ be (some) its $\Gt$-closure.
To show that $\bG$ is a line graph of a multigraph, by Theorem~\ref{thmA-BeMe},
we show that $\bG$ does not contain as an induced subgraph any of the graphs 
$G_1,\ldots,G_7$ of Figure~\ref{fig-BeMe}.
If $G$ is Hamilton-connected, then $\bG$ is complete and 
the statement is trivial. So, assume that $G$ (and hence also $\bG$) 
is not Hamilton-connected. 

Since $G_1\iso\claw$ and $\bG$ was obtained from $G$ by a series of local 
completions, obviously $\bG$ is $G_1$-free. We prove the following 
three facts (for the graphs $W_5$, $W_4$, $P_6^2$ and $P_6^{2+}$ see 
Fig.~\ref{fig-kola_a_cesty}):
\begin{mathitem}
 \item[$\bullet$] $\bG$ is $W_5$-free,
 \item[$\bullet$] $\bG$ is $W_4$-free,
 \item[$\bullet$] $\bG$ is $\{P_6^2,P_6^{2+}\}$-free.
\end{mathitem}
This will establish the result since $G_3\iso W_5$, each of the graphs 
$G_5,G_6,G_7$ contains an induced $W_4$, $G_2\iso P_6^2$, and 
$G_4\iso P_6^{2+}$.

%
%
\begin{figure}[ht]
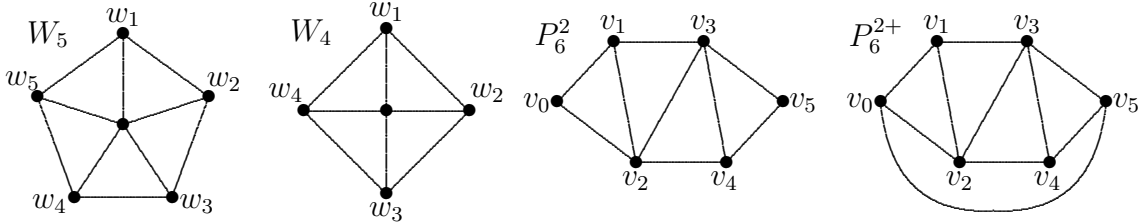

$$\bp
\setcoordinatesystem units <1mm,1mm>
\setplotarea x from -50 to 50, y from -7 to 7
\put{\beginpicture
\setcoordinatesystem units <.1mm,.1mm>
\setplotarea x from -120 to 120, y from -120 to 120
\put{$\bullet$} at     0   0
\put{$\bullet$} at     0 120
\put{$\bullet$} at   114  37
\put{$\bullet$} at    65 -97
\put{$\bullet$} at   -65 -97
\put{$\bullet$} at  -114 37
\put{$w_1$} at     0  143
\put{$w_2$} at   134   57
\put{$w_3$} at    98 -105
\put{$w_4$} at   -98 -105
\put{$w_5$} at  -134   58
\plot  0 120   114 37  65 -97  -65 -97  -114 37   0 120 /
\plot  0 0    0 120 /
\plot  0 0  114  37 /
\plot  0 0   65 -97 /
\plot  0 0  -65 -97 /
\plot  0 0 -114  37 /
\put{$W_5$} at  -100 120
\endpicture} at -65  0
\put{\beginpicture
\setcoordinatesystem units <.1mm,.1mm>
\setplotarea x from -120 to 120, y from -120 to 120
\put{$\bullet$} at     0    0
\put{$\bullet$} at     0  110
\put{$\bullet$} at   110    0
\put{$\bullet$} at     0 -110
\put{$\bullet$} at  -110    0
\put{$w_1$} at     0  133
\put{$w_2$} at   134   25
\put{$w_3$} at     0 -133
\put{$w_4$} at  -134   25
\plot  0 110   110 0  0 -110  -110 0  0 110 /
\plot  0 0    0  110 /
\plot  0 0  110    0 /
\plot  0 0    0 -110 /
\plot  0 0 -110    0 /
\put{$W_4$} at  -100 105
\endpicture} at -30  0
\put{
\bp
\setcoordinatesystem units <0.3mm,.8mm>
\setplotarea x from -20 to 20, y from -15 to 10
\put{$\bullet$} at  -50   0
\put{$\bullet$} at  -25  10
\put{$\bullet$} at  -15 -10
\put{$\bullet$} at   15  10
\put{$\bullet$} at   25 -10
\put{$\bullet$} at   50   0
\plot -50 0 -25 10  15 10  50 0  25 -10 -15 -10  -50 0 /
\plot -25 10   -15 -10  15 10  25 -10 /
\put{$v_0$} at  -59   0
\put{$v_1$} at  -25  13
\put{$v_2$} at  -15 -13
\put{$v_3$} at   15  13
\put{$v_4$} at   24 -13
\put{$v_5$} at   59   0
\put{$P_6^{2}$} at -52 11
\ep} at  7 1
\put{
\bp
\setcoordinatesystem units <0.3mm,.8mm>
\setplotarea x from -20 to 20, y from -15 to 10
\put{$\bullet$} at  -50   0
\put{$\bullet$} at  -25  10
\put{$\bullet$} at  -15 -10
\put{$\bullet$} at   15  10
\put{$\bullet$} at   25 -10
\put{$\bullet$} at   50   0
\plot -50 0 -25 10  15 10  50 0  25 -10 -15 -10  -50 0 /
\plot -25 10   -15 -10  15 10  25 -10 /
\put{$v_0$} at  -59   0
\put{$v_1$} at  -25  13
\put{$v_2$} at  -15 -13
\put{$v_3$} at   15  13
\put{$v_4$} at   24 -13
\put{$v_5$} at   59   0
\setquadratic
\plot -50 0  -35 -14   0 -18   35 -14   50 0 /
\setlinear
\put{$P_6^{2+}$} at -52 11
\ep} at   50 1
\ep$$
\bsm
\caption{The 5-wheel $W_5$, the 4-wheel $W_4$, and the graphs $P_6^2$ and 
        $P_6^{2+}$}
\label{fig-kola_a_cesty}
\end{figure}

This will be done in the following Propositions~\ref{prop-closure-W5},
\ref{prop-closure-W4} and \ref{prop-closure-P6}.

\ms

Throughout the proof, when listing vertices of an induced $W_5$ or $W_4$,
we will always list the center first, and we will list vertices of a $P_6^2$
or a $P_6^{2+}$ in the order indicated by the indices in 
Fig.~\ref{fig-kola_a_cesty}.

%
%
\begin{proposition}
\label{prop-closure-W5}
Let $G$ be a $\{\claw,\Gt\}$-free graph and let $\bG$ be its $\Gt$-closure.
Then $\bG$ is $W_5$-free.
\end{proposition}

\begin{proof}
If $x$ is a center of an induced $W_5$ in $\bG$, then $x$ is locally 
2-connected in $\bG$ by Lemma~\ref{lemmaA-2_indep_sets},
hence $x$ is feasible in $\bG$ by Theorem~\ref{thmA-HC-2-conn_neighb}.
Thus, by the definition of the $\Gt$-closure, $\bGstx$ contains an induced 
$\Gt$. To reach a contradiction, we show that this is not possible. 

We first prove several lemmas.

%
%
\begin{lemma}
\label{lemma-W5-neighbors}
Let $G$ be a claw-free graph and $F\indsub G$, $F\iso W_5$, with center $x$ 
and cycle $C$. Then
\begin{mathitem}
 \item every vertex $y\in N_G(x)\sm V(F)$ has at least three consecutive
    neighbors on $C$,
 \item for every $y_1,y_2\in N_G(x)$, $\dist_{\la N_G(x)\rag}(y_1,y_2)\leq 2$,
 \item for every $y\in N_G(x)\sm V(F)$ that is not in an induced $C_5$
    in $\la N_G(x)\rag$, every induced $C_5$ in $\la N_G(x)\rag$ contains 
    at least 4 neighbors of $y$.
\end{mathitem}
\end{lemma}

\begin{proof} Let $F=\lab w,w_1,w_2,w_3,w_4,w_5\rabg\iso W_5$.

\ms

$(i)$ If $N_G(y)\cap V(C)=\emptyset$, then $x$ is a center of 
a claw; thus, say, $w_1y\in E(G)$. Since $\lab w_1,y,w_2,w_5\rabg\niso\claw$, 
by symmetry, $w_2y\in E(G)$. If $y$ is adjacent to neither $w_3$ nor $w_5$,
then $\lab x,w_3,w_5,y\rabg\iso\claw$; thus, by symmetry, say, 
$yw_3\in E(G)$.

\ms

$(ii)$ By $(i)$, any $y_1,y_2\in N_G(x)$ are adjacent or have a common 
neighbor in $N_G(x)$. 

\ms 

$(iii)$ Let $C'=w'_1w'_2w'_3w'_4w'_5w'_1$ be an induced $C_5$ in $N_G(x)$. 
By $(i)$, $y$ has at least 3 con\-secutive neighbors on $C'$. If 
$|N_G(y)\cap V(C')|\geq 4$, we are done, thus, let, say, 
$N_G(y)\cap V(C')=\{w'_1,w'_2,w'_3\}$. But then 
$\lab w'_1, y, w'_3,w'_4,w'_5\rabg\iso C_5$, contradicting the assumption.
\end{proof}

%
%
\begin{lemma}
\label{lemma-W5-dist3}
Let $i\geq 1$, let $G$ be a $\{\claw,\Gi\}$-free graph, and let 
$x\in V_{EL}(G)$. Let $F\indsub \Gstx$ be such that $F\iso\Gi$ and a 
triangle of $F$ contains a new edge $y_1y_2$. 
Then $\dist_{\la N_G(x)\rag}(y_1,y_2)=3$.
\end{lemma}

\begin{proof}
Let $F\indsub\Gstx$, $F\iso\Gi$, and let $y_1y_2$ be a new edge in some 
of its triangles. We use the labeling of the vertices of $F$ as in 
Fig.~\ref{fig-special_graphs}$(d)$.
Since $\la N_G(x)\ra_{\Gstx}$ is a clique, all new edges in $F$ are in one 
of its triangles, say, $t_1t_2p_1t_1$. If $t_1t_2$ is the only new edge, 
then $\lab p_1,p_2,t_1,t_2\rabg\iso\claw$, and if all three edges are new, 
then $\lab x,t_1,t_2,p_1\rabg\iso\claw$. Thus, by symmetry, we can choose 
the notation such that the edge $t_1p_1$ (and possibly also one of the 
edges $t_2p_1$, $t_1t_2$) is new, i.e., $y_1=t_1$ and $y_2=p_1$. 

Note that $2\leq\dist_{\la N_G(x)\rag}(t_1,p_1)\leq 3$ since $t_1p_1$ is a 
new edge and $x$ is not the center of an induced claw.
Suppose, to the contrary, that $\dist_{N_G(x)}(t_1,p_1)=2$.
If $xt_2,p_1t_2\in E(G)$, then 
$\lab x,t_2,p_1\car p_2,\ldots, p_{i+1},t_3,t_4\rabg\iso\Gi$,
a contradiction. Thus, by the assumption that 
$\dist_{\la N_G(x)\rag}(t_1,p_1)=2$,
there is a vertex $z\in N_G(x)$ such that $zt_1,zp_1\in E(G)$ (note that
possibly $x=t_2$). Since 
$\lab z,x,p_1,p_2,\ldots,p_{i+1},t_3,t_4\rabg\niso\Gi$, $zw\in E(G)$
for some $w\in\{p_2,\ldots,p_{i+1},t_3,t_4\}$. However, if 
$w\in\{p_3,\ldots,p_{i+1},t_3,t_4\}$, then $\lab z,t_1,p_1,w\rabg\iso\claw$.
Hence $zw\notin E(G)$ for $w\in\{p_3,\ldots,p_{i+1},t_3,t_4\}$, implying
$zp_2\in E(G)$, and then 
$\lab t_1,x,z,p_2,\ldots,p_{i+1},t_3,t_4\rabg\iso\Gi$, a contradiction.
\end{proof}

%
%
\begin{lemma}
\label{lemma-W5-new-in-path}
Let $i\geq 2$ and $1\leq j\leq i$.
Let $G$ be a $\{\claw,\Gi\}$-free graph and $x\in V_{NS}(G)$ 
such that $\Gstx$ contains an induced subgraph $F\iso\Gi$ with a new
edge $p_jp_{j+1}$ in the path of $F$. For $k=1,2$, let $z_k\in V(G)$ be
such that $\{ p_j,p_{j+1},x\}\subset N_G(z_k)$. 
Then $z_1z_2\in E(G)$. 
\end{lemma}

\begin{proof}
Let, to the contrary, $z_1z_2\notin E(G)$. By symmetry, we can choose 
the notation such that $1\leq j\leq i-1$ (i.e., $p_{j+1}$ is not the 
last vertex of the path of $F$). 
Since $\lab p_{j+1},z_1,z_2,p_{j+2}\rabg\niso\claw$, we have, up to a 
symmetry, $z_2p_{j+2}\in E(G)$. Since each of the edges $z_2w$, 
$w\in\{t_1,t_2,p_1,\ldots,p_{j-1},p_{j+3},\ldots,p_{i+1},t_3,t_4\}$
yields an induced $\claw$ in $G$ with center at $z_2$, we have 
$\lab t_1,t_2,p_1,\ldots,p_j,z_2,p_{j+2},\ldots,p_{i+1},t_3,t_4\rabg\iso\Gi$, 
a contradiction.
\end{proof}

Suppose now, to the contrary, that $\bG$ contains an induced subgraph 
$W\iso W_5$, set $W=\lab x,w_1,w_2,w_3,w_4,w_5\rabgb$, and let 
$F\indsub \bGstx$ be such that $F\iso\Gt$. Set 
$F=\lab t_1,t_2,p_1,p_2,p_3\car p_4,t_3,t_4\rabgb$. Since $\bG$ is 
$\Gt$-free, at least one edge of $F$ is new.

If there is a new edge $y_1y_2$ in a triangle of $F$, then we have 
$\dist_{\bG}(y_1,y_2)=2$ by Lemma~\ref{lemma-W5-neighbors}$(ii)$, and 
$\dist_{\bG}(y_1,y_2)=3$ by Lemma~\ref{lemma-W5-dist3}, a contradiction.
Thus, a new edge is on the path of $F$, and, moreover, $F$ has exactly one 
new edge since it is induced and $x$ is simplicial in $\bGstx$. 

\ms

Choose $W\iso W_5$ so that $|V(W)\cap V(F)|$ is maximized. 
Clearly, $|V(W)\cap V(F)|\leq 2$, and we can choose the notation such 
that the new edge is $p_1p_2$ or $p_2p_3$.

\begin{mylist}
  \item[\underline{\bf Case 1:}] $|V(W)\cap V(F)|=2$.
  \begin{mylist}
    \item[\underline{\bf Subcase 1.1:}] {\sl The edge $p_1p_2$ is new.} \quad \\
     We can choose the notation such that $p_1=w_2$ and $p_2=w_5$.
     First observe that since 
     $\lab w_2,w_3,x,w_5,p_3,p_4,t_3,t_4\rabgb\niso\Gt$, $w_3z\in E(\bG)$
     for some $z\in \{p_3,p_4,t_3,t_4\}$. 
     Since $\lab w_2,t_1,w_1,w_3\rabgb\niso\claw$, $t_1w_1\in E(\bG)$ or 
     $t_1w_3\in E(\bG)$. But if $t_1w_3\in E(\bG)$, then 
     $\lab w_3,t_1,x,z\rabgb\iso\claw$, a contradiction. Hence 
     $t_1w_1\in E(\bG)$. Since 
     $\lab t_1,w_2,w_1,w_5,p_3\car p_4\car t_3,t_4\rabgb\niso\Gt$, we have 
     $w_1z'\in E(\bG)$ for some $z'\in\{p_3,p_4,t_3,t_4\}$, but for each 
     of these possibilities, $\lab w_1,t_1,x,z'\rabgb\iso\claw$, 
     a contradiction.
    \item[\underline{\bf Subcase 1.2:}] {\sl The edge $p_2p_3$ is new.} \quad \\
     Choose the notation such that $p_2=w_2$ and $p_3=w_5$. Considering 
     $\lab w_2,w_1,w_3,p_1\rabgb\niso\claw$, we have $p_1w_1\in E(\bG)$ or
     $p_1w_3\in E(\bG)$. Let first $p_1w_1\in E(\bG)$. Since 
     $\lab t_1,t_2,p_1,w_1\car w_3,p_4,t_3,t_4\rabgb\niso\Gt$, $w_1z\in E(\bG)$
     for some $z\in\{t_1,t_2\}\cup\{p_4,t_3,t_4\}$; however, if 
     $z\in\{t_1,t_2\}$, then $\lab w_1,z,w_2,w_5\rabgb\iso\claw$, and if 
     $z\in\{p_4,t_3,t_4\}$, then $\lab w_1,z,x,p_1\rabgb\iso\claw$. 
     Hence $p_1w_3\in E(\bG)$. Symmetrically, $p_4w_4\in E(\bG)$. 
     Since $\lab t_1,t_2,p_1,w_3,w_4\car p_4,t_3,t_4\rabgb\niso\Gt$, 
     $w_3$ or $w_4$ has a neighbor among the vertices $t_1,t_2,t_3,t_4$.
     However, if $w_3z\in E(\bG)$ for $z\in\{t_1,t_2\}$, then 
     $\lab w_3,z,w_2,w_4\rabgb\iso\claw$, and if $w_3z\in E(\bG)$ for 
     $z\in\{t_3,t_4\}$, then $\lab w_3,p_1,x,z\rabgb\iso\claw$; the 
     situation for $w_4$ is symmetric. Thus,   
     $\lab t_1,t_2,p_1,w_3,w_4,p_4\car t_3,t_4\rabgb\iso\Gt$, 
     a contradiction.
  \end{mylist}
  \item[\underline{\bf Case 2:}] $|V(W)\cap V(F)|=1$. \quad \\
   Let the new edge be $p_ip_{i+1}$, $i\in\{1,2\}$, and, by symmetry, choose 
   the notation such that $p_{i+1}=w_1$ (note that our proof of this case 
   does not use the rest of $F$ and hence it is symmetric also for $i=1$). 
   Since $\lab x,p_i,w_2,w_5\rabgb\niso\claw$, by symmetry, we have 
   $p_iw_5\in E(\bG)$, and, by Lemma~\ref{lemma-W5-new-in-path}, 
   $p_iw_2\notin E(\bG)$. Since $\lab x,p_i,w_1,w_3\rabgb\niso\claw$, 
   $p_iw_3\in E(\bG)$. Then the subgraph 
   $W'=\lab x,w_1,w_2,w_3,p_i,w_5\rabgb\iso W_5$ is an induced $W_5$ in~$\bG$
   with $|V(W')\cap V(F)|=2$, contradicting the choice of $W$.
  \item[\underline{\bf Case 3:}] $V(W)\cap V(F)=\emptyset$. \quad \\
   Let again $p_ip_{i+1}$, $i\in\{1,2\}$, be the new edge in $F$. By 
   Lemma~\ref{lemma-W5-neighbors}$(iii)$, each of $p_i,p_{i+1}$ has at least 
   4 neighbors on the 5-cycle $C=W-x$. By the pigeonhole principle, $p_i$
   and $p_{i+1}$ have at least 3 common neighbors on $C$. Since $C$ is 
   induced, at least two of these common neighbors are nonadjacent, 
   contradicting Lemma~\ref{lemma-W5-new-in-path}.
\end{mylist}
\bsm\bsm
\end{proof}

Note that now we know that $\bG$ is $\{\claw,W_5\}$-free. 
Since $\bG$ is 3-connected and not Hamilton-connected, we have $\alpha(\bG)\geq 3$
by Theorem~\ref{thmA-Chvatal-Erdos}, hence for every $v\in V(G)$, 
$\NbG(v)$ can be covered by two cliques. This fact will be useful in the next proof.

%
%
\begin{proposition}
\label{prop-closure-W4}
Let $G$ be a 3-connected $\{\claw,\Gt\}$-free graph and let $\bG$ be its $\Gt$-closure.
Then $\bG$ is $W_4$-free.
\end{proposition}

\begin{proof}
Similarly as before, if $x\in V(G)$ is a center of an induced $W_4$ in 
$\bG$, then $x$ is locally 2-connected in $\bG$ by 
Lemma~\ref{lemmaA-2_indep_sets}, hence $x$ is feasible in $\bG$ by 
Theorem~\ref{thmA-HC-2-conn_neighb}, and $\bGstx$ contains an induced 
$\Gt$. We again show that this is not possible.

%
%
\begin{lemma}
\label{lemma-W4-substructure}
Let $G$ be a claw-free graph, let $W$ be an induced subgraph of $G$
such that $W=\lab x,w_1,w_2,w_3,w_4\rabg\iso W_4$, let $z_1,z_2\in N_G(x)$ 
be such that $\dist_{\la N_G(x)\rag}(z_1,z_2)=3$,
and let $R$ be the graph shown in Fig.~\ref{fig-W4-substructure}.
Then $G$ contains $R$ as an induced subgraph.
\end{lemma}

%
%
\begin{figure}[ht]
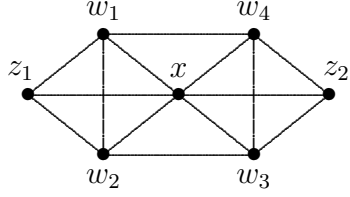

$$\bp
\setcoordinatesystem units <1mm,1mm>
\setplotarea x from -10 to 10, y from -7 to 7
\put{\beginpicture
\setcoordinatesystem units <1mm,0.8mm>
\setplotarea x from -12 to 12, y from -12 to 12
\put{$\bullet$} at     0    0
\put{$\bullet$} at    10   10
\put{$\bullet$} at    10  -10
\put{$\bullet$} at   -10  -10
\put{$\bullet$} at   -10   10
\put{$\bullet$} at    20    0
\put{$\bullet$} at   -20    0
\put{$x$} at     0  4
\put{$w_1$} at    -10  14
\put{$w_2$} at    -10 -14
\put{$w_3$} at     10 -14
\put{$w_4$} at     10  14
\put{$z_1$} at    -21   4
\put{$z_2$} at     21   4
\plot  -20 0  -10 10  10 10  20 0  10 -10  -10 -10  -20 0 /
\plot  -10 -10  -10 10  10 -10  10 10  -10 -10 /
\plot  -20 0  20 0 /
\endpicture} at 0  0
\ep$$
\bsm
\caption{The graph $R$ used in Lemma~\ref{lemma-W4-substructure}.}
\label{fig-W4-substructure}
\end{figure}

\begin{proof}
Let $C=W-x$ be the 4-cycle $C=w_1w_2w_3w_4$ of the 4-wheel $W$. 
First observe that $\dist_{\la N_G(x)\rag}(z_i,C)\leq 1$ since otherwise
e.g. $\lab x,z_i,w_1,w_3\rabg\iso\claw$, $i=1,2$. 
On the other hand, if some $z_i$ is on $C$, say, $z_1=w_1$, then 
$w_2z_2,w_4z_2\notin E(G)$ by the assumption that
$\dist_{\la N_G(x)\rag}(z_1,z_2)=3$, and then 
$\lab x,w_2,w_4,z_2\rabg\iso\claw$. Hence $\dist_{\la N_G(x)\rag}(z_i,C)=1$,
$i=1,2$. 

Up to a symmetry, let $z_1w_1\in E(G)$. Since 
$\lab x,w_2,w_4,z_2\rabg\niso\claw$, up to a symmetry, $w_4z_2\in E(G)$. 
Since $\lab x,z_1,w_2,w_4\rabg\niso\claw$ and $z_1w_4\notin E(G)$ by the 
assumption $\dist_{\la N_G(x)\rag}(z_1,z_2)=3$, we have $z_1w_2\in E(G)$. 
Symmetrically, $w_3z_2\in E(G)$. 
Thus, we have the graph $R$.
Finally, the subgraph is induced since any additional edge would 
contradict either the assumption $\dist_{\la N_G(x)\rag}(z_1,z_2)=3$,
or the fact that $\lab x,w_1,w_2,w_3,w_4\rabg\iso W_4$.
\end{proof}

Let now, to the contrary, $W=\lab x,w_1,w_2,w_3,w_4\rabgb\iso W_4$
be an induced subgraph of $\bG$, and let $F\indsub\bGstx$ be such that 
$F\iso\Gt$. Set again $F=\lab t_1,t_2,p_1,p_2,p_3,p_4,t_3,t_4\rabgb$.
Since $\bG$ is $\Gt$-free, at least one edge of $F$ is new. Since
$x\in V_{SI}(\bGstx)$, either all new edges are in one of the triangles 
of $F$, or the only new edge of $F$ is some edge on the path.

\begin{mylist}
 \item[\underline{\bf Case 1:}] {\sl New edges are in a triangle of $F$.} \quad \\
  We can choose the notation such that the new edges are in the triangle 
  $p_1t_1t_2p_1$. As before, if $t_1t_2$ is the only new edge in $F$, then 
  $\lab p_1,t_1,t_2,p_2\rabgb\iso\claw$; hence we can choose the notation 
  such that $t_1p_1$ (and possibly also one of $t_1t_2$, $t_2p_1$) is new.
  By Lemma~\ref{lemma-W5-dist3},  $\dist_{\la\NbG(x)\ragb}(t_1,p_1)=3$.
  By Lemma~\ref{lemma-W4-substructure}, $p_1,t_1$ and $W$ are in the 
  induced subgraph $R$ of Fig.~\ref{fig-W4-substructure}.
  Thus, in all the cases (independently of whether the edges $t_1t_2$, $t_2p_1$
  are present in $\bG$ or not), $\bG$ contains the subgraph $F_0$ shown 
  in Fig.~\ref{fig-W4-subgraph_F_0}.

%
%
\begin{figure}[ht]
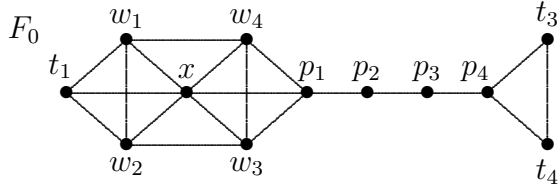

$$\bp
\setcoordinatesystem units <1mm,1mm>
\setplotarea x from -10 to 10, y from -7 to 7
\put{\beginpicture
\setcoordinatesystem units <.8mm,.7mm>
\setplotarea x from -12 to 12, y from -12 to 12
\put{$\bullet$} at     0    0
\put{$x$} at     0  4
\put{$\bullet$} at   -20    0
\put{$t_1$} at    -21   5
\put{$\bullet$} at   -10   10
\put{$w_1$} at    -10  14
\put{$\bullet$} at   -10  -10
\put{$w_2$} at    -10 -14
\put{$\bullet$} at    10  -10
\put{$w_3$} at     10 -14
\put{$\bullet$} at    10   10
\put{$w_4$} at     10  14
\put{$\bullet$} at    20    0
\put{$p_1$} at     21   4
\put{$\bullet$} at    30    0
\put{$p_2$} at     30   4
\put{$\bullet$} at    40    0
\put{$p_3$} at     40   4
\put{$\bullet$} at    50    0
\put{$p_4$} at     48   4
\put{$\bullet$} at    60   10
\put{$t_3$} at     60  15
\put{$\bullet$} at    60  -10
\put{$t_4$} at     60 -15
\plot  -20 0  -10 10  10 10  20 0  10 -10  -10 -10  -20 0 /
\plot  -10 -10  -10 10  10 -10  10 10  -10 -10 /
\plot  -20 0  20 0 /
\plot 20 0   50 0  60 10  60 -10  50 0 /
\put{$F_0$} at    -27   12
\endpicture} at 20  0
\ep$$
\bsm
\caption{The subgraph $F_0$}
\label{fig-W4-subgraph_F_0}
\bsm
\end{figure}

Since e.g. $\lab x,w_3,p_1,p_2,p_3,p_4,t_3,t_4\rab_{F_0}\niso\Gt$, $F_0$ is not
an induced subgraph of $\bG$.
Checking by computer all possible sets of additional edges such that the 
resulting graph is still $\{\claw,\Gt,W_5\}$-free, we obtain 10 exceptional graphs 
$F_1,\ldots,F_{10}$ (see \cite{computing}). 
For each of them, 
$V(F_i)=V(F_0)$, $i=1,\ldots,10$, and their edge sets are:\\
$E(F_1)=E(F_0)\cup\{w_1t_3,w_2t_4,w_3t_4,w_4t_3\}$,\\
$E(F_2)=E(F_0)\cup\{w_1t_3,w_2p_2,w_2p_3,w_3p_2,w_3p_3,w_4t_3\}$,\\
$E(F_3)=E(F_0)\cup\{w_1t_3,w_2p_3,w_2p_4,w_3p_3,w_3p_4,w_4t_3\}$,\\
$E(F_4)=E(F_0)\cup\{w_1p_2,w_1p_3,w_2p_3,w_2p_4,w_3p_3,w_3p_4,w_4p_2,w_4p_3\}$,\\
$E(F_5)=E(F_0)\cup\{w_1p_2,w_1p_3,w_2t_3,w_2t_4,w_3t_3,w_3t_4,w_4p_2,w_4p_3\}$,\\
$E(F_6)=E(F_0)\cup\{w_1p_3,w_1p_4,w_2t_3,w_2t_4,w_3t_3,w_3t_4\car w_4p_3,w_4p_4\}$,\\
$E(F_7)=E(F_0)\cup\{w_1p_4,w_1t_3,w_1t_4,w_2t_4,w_3t_4,w_4p_4,w_4t_3,w_4t_4\}$,\\
$E(F_8)=E(F_0)\cup
\{w_1p_2,w_1p_3,w_2p_4,w_2t_3,w_2t_4,w_3p_4,w_3t_3,w_3t_4,w_4p_2,w_4p_3\}$,\\
$E(F_9)=E(F_0)\cup
\{w_1p_3,w_1p_4,w_2p_4,w_2t_3,w_2t_4,w_3p_4,w_3t_3,w_3t_4,w_4p_3,w_4p_4\}$,\\
$E(F_{10})=E(F_0)\cup
\{w_1p_4,w_1t_3,w_1t_4,w_2t_3,w_2t_4,w_3t_3,w_3t_4,w_4p_4,w_4t_3,w_4t_4\}$.

The graphs $F_1,\ldots,F_{10}$ are also shown in Fig.~\ref{fig-W4-10_exceptions}
(where the double-circled vertices indicate the weakly feasible sets given
by Claim~\ref{claim-W4-3}).

\newcommand{\x}{.453mm}
\newcommand{\y}{.7mm}

%
%
\begin{figure}[ht]
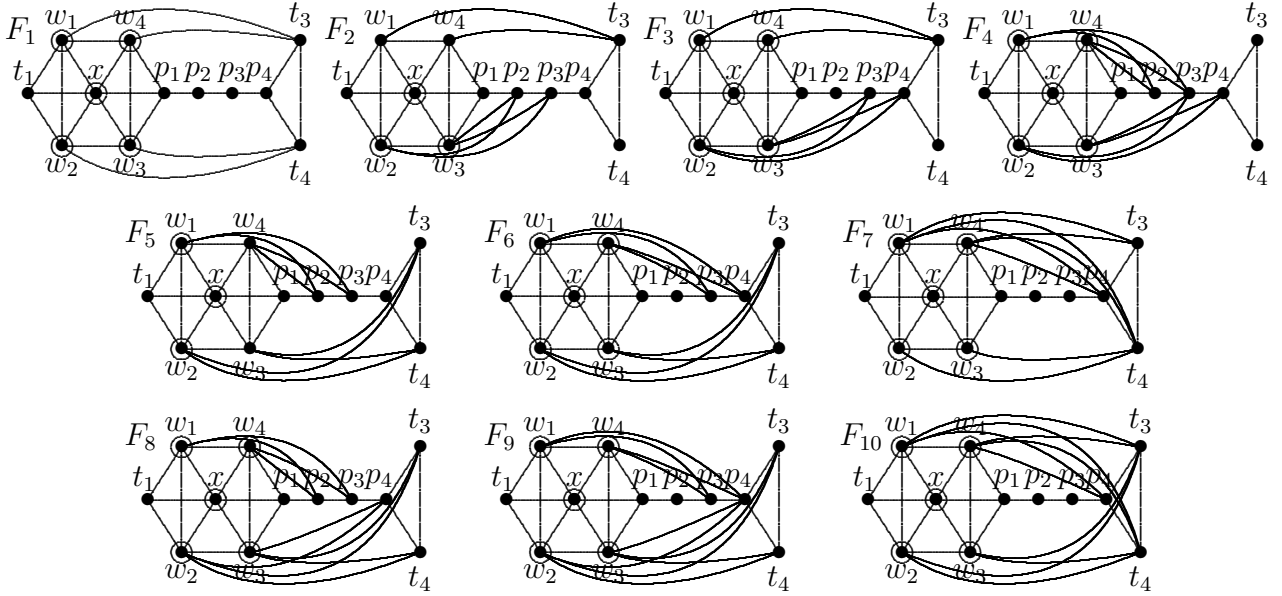

$$\bp
\setcoordinatesystem units <1.06mm,0.9mm>
\setplotarea x from -10 to 10, y from -7 to 7
\put{\beginpicture
\setcoordinatesystem units <\x,\y>
\setplotarea x from -12 to 12, y from -12 to 12
\put{$\bullet$} at     0    0
\circulararc  360  degrees from   0  2 center at  0   0
\put{$x$} at     0  4
\put{$\bullet$} at   -20    0
\put{$t_1$} at    -21   4
\put{$\bullet$} at   -10   10
\circulararc  360  degrees from  -10  12 center at  -10   10
\put{$w_1$} at    -10  14
\put{$\bullet$} at   -10  -10
\circulararc  360  degrees from  -10 -12 center at  -10  -10
\put{$w_2$} at    -10 -14
\put{$\bullet$} at    10  -10
\circulararc  360  degrees from   10 -12 center at   10  -10
\put{$w_3$} at     10 -14
\put{$\bullet$} at    10   10
\circulararc  360  degrees from   10  12 center at   10   10
\put{$w_4$} at     10  14
\put{$\bullet$} at    20    0
\put{$p_1$} at     21   4
\put{$\bullet$} at    30    0
\put{$p_2$} at     30   4
\put{$\bullet$} at    40    0
\put{$p_3$} at     40   4
\put{$\bullet$} at    50    0
\put{$p_4$} at     48   4
\put{$\bullet$} at    60   10
\put{$t_3$} at     60  15
\put{$\bullet$} at    60  -10
\put{$t_4$} at     60 -15
\plot  -20 0  -10 10  10 10  20 0  10 -10  -10 -10  -20 0 /
\plot  -10 -10  -10 10  10 -10  10 10  -10 -10 /
\plot  -20 0  20 0 /
\plot 20 0   50 0  60 10  60 -10  50 0 /
\setquadratic
\plot -10  10   20  16    60  10 /  
\plot  10  10   30  12    60  10 /  
\plot -10 -10   20 -16    60 -10 /  
\plot  10 -10   30 -12    60 -10 /  
\setlinear
\put{$F_1$} at    -22  12
\endpicture} at -60  30
\put{\beginpicture
\setcoordinatesystem units <\x,\y>
\setplotarea x from -12 to 12, y from -12 to 12
\put{$\bullet$} at     0    0
\circulararc  360  degrees from   0  2 center at  0   0
\put{$x$} at     0  4
\put{$\bullet$} at   -20    0
\put{$t_1$} at    -21   4
\put{$\bullet$} at   -10   10
\put{$w_1$} at    -10  14
\put{$\bullet$} at   -10  -10
\circulararc  360  degrees from  -10 -12 center at  -10  -10
\put{$w_2$} at    -10 -14
\put{$\bullet$} at    10  -10
\circulararc  360  degrees from   10 -12 center at   10  -10
\put{$w_3$} at     10 -14
\put{$\bullet$} at    10   10
\put{$w_4$} at     10  14
\put{$\bullet$} at    20    0
\put{$p_1$} at     21   4
\put{$\bullet$} at    30    0
\put{$p_2$} at     30   4
\put{$\bullet$} at    40    0
\put{$p_3$} at     40   4
\put{$\bullet$} at    50    0
\put{$p_4$} at     48   4
\put{$\bullet$} at    60   10
\put{$t_3$} at     60  15
\put{$\bullet$} at    60  -10
\put{$t_4$} at     60 -15
\plot  -20 0  -10 10  10 10  20 0  10 -10  -10 -10  -20 0 /
\plot  -10 -10  -10 10  10 -10  10 10  -10 -10 /
\plot  -20 0  20 0 /
\plot 20 0   50 0  60 10  60 -10  50 0 /
\setquadratic
\plotsymbolspacing=0.1pt
\plot -10  10   20  16    60  10 /  
\plot  10  10   30  12    60  10 /  
\plot -10 -10   15 -10.5  30   0 /  
\plot -10 -10   20 -11    40   0 /  
\plot  10 -10   20  -5    30   0 /  
\plot  10 -10   25  -6    40   0 /  
\setlinear
\put{$F_2$} at    -22  12
\endpicture} at  -20  30
\put{\beginpicture
\setcoordinatesystem units <\x,\y>
\setplotarea x from -12 to 12, y from -12 to 12
\put{$\bullet$} at     0    0
\circulararc  360  degrees from   0  2 center at  0   0
\put{$x$} at     0  4
\put{$\bullet$} at   -20    0
\put{$t_1$} at    -21   4
\put{$\bullet$} at   -10   10
\circulararc  360  degrees from  -10  12 center at  -10   10
\put{$w_1$} at    -10  14
\put{$\bullet$} at   -10  -10
\circulararc  360  degrees from  -10 -12 center at  -10  -10
\put{$w_2$} at    -10 -14
\put{$\bullet$} at    10  -10
\circulararc  360  degrees from   10 -12 center at   10  -10
\put{$w_3$} at     10 -14
\put{$\bullet$} at    10   10
\circulararc  360  degrees from   10  12 center at   10   10
\put{$w_4$} at     10  14
\put{$\bullet$} at    20    0
\put{$p_1$} at     21   4
\put{$\bullet$} at    30    0
\put{$p_2$} at     30   4
\put{$\bullet$} at    40    0
\put{$p_3$} at     40   4
\put{$\bullet$} at    50    0
\put{$p_4$} at     48   4
\put{$\bullet$} at    60   10
\put{$t_3$} at     60  15
\put{$\bullet$} at    60  -10
\put{$t_4$} at     60 -15
\plot  -20 0  -10 10  10 10  20 0  10 -10  -10 -10  -20 0 /
\plot  -10 -10  -10 10  10 -10  10 10  -10 -10 /
\plot  -20 0  20 0 /
\plot 20 0   50 0  60 10  60 -10  50 0 /
\setquadratic
\plotsymbolspacing=0.1pt
\plot -10  10   20  16    60  10 /  
\plot  10  10   30  12    60  10 /  
\plot -10 -10   20 -11    40   0 /  
\plot -10 -10   20 -12    50   0 /  
\plot  10 -10   25  -6    40   0 /  
\plot  10 -10   25  -7    50   0 /  
\setlinear
\put{$F_3$} at    -22  12
\endpicture} at  20  30
\put{\beginpicture
\setcoordinatesystem units <\x,\y>
\setplotarea x from -12 to 12, y from -12 to 12
\put{$\bullet$} at     0    0
\circulararc  360  degrees from   0  2 center at  0   0
\put{$x$} at     0  4
\put{$\bullet$} at   -20    0
\put{$t_1$} at    -21   4
\put{$\bullet$} at   -10   10
\circulararc  360  degrees from  -10  12 center at  -10   10
\put{$w_1$} at    -10  14
\put{$\bullet$} at   -10  -10
\circulararc  360  degrees from  -10 -12 center at  -10  -10
\put{$w_2$} at    -10 -14
\put{$\bullet$} at    10  -10
\circulararc  360  degrees from   10 -12 center at   10  -10
\put{$w_3$} at     10 -14
\put{$\bullet$} at    10   10
\circulararc  360  degrees from   10  12 center at   10   10
\put{$w_4$} at     10  14
\put{$\bullet$} at    20    0
\put{$p_1$} at     21   4
\put{$\bullet$} at    30    0
\put{$p_2$} at     30   4
\put{$\bullet$} at    40    0
\put{$p_3$} at     40   4
\put{$\bullet$} at    50    0
\put{$p_4$} at     48   4
\put{$\bullet$} at    60   10
\put{$t_3$} at     60  15
\put{$\bullet$} at    60  -10
\put{$t_4$} at     60 -15
\plot  -20 0  -10 10  10 10  20 0  10 -10  -10 -10  -20 0 /
\plot  -10 -10  -10 10  10 -10  10 10  -10 -10 /
\plot  -20 0  20 0 /
\plot 20 0   50 0  60 10  60 -10  50 0 /
\setquadratic
\plotsymbolspacing=0.1pt
\plot -10  10   15  10.5  30   0 /  
\plot -10  10   20  11    40   0 /  
\plot  10  10   20   5    30   0 /  
\plot  10  10   25   6    40   0 /  
\plot -10 -10   20 -11    40   0 /  
\plot -10 -10   20 -12    50   0 /  
\plot  10 -10   25  -6    40   0 /  
\plot  10 -10   25  -7    50   0 /  
\setlinear
\put{$F_4$} at    -22  12
\endpicture} at   60 30
\put{\beginpicture
\setcoordinatesystem units <\x,\y>
\setplotarea x from -12 to 12, y from -12 to 12
\put{$\bullet$} at     0    0
\circulararc  360  degrees from   0  2 center at  0   0
\put{$x$} at     0  4
\put{$\bullet$} at   -20    0
\put{$t_1$} at    -21   4
\put{$\bullet$} at   -10   10
\circulararc  360  degrees from  -10  12 center at  -10   10
\put{$w_1$} at    -10  14
\put{$\bullet$} at   -10  -10
\circulararc  360  degrees from  -10 -12 center at  -10  -10
\put{$w_2$} at    -10 -14
\put{$\bullet$} at    10  -10
\put{$w_3$} at     10 -14
\put{$\bullet$} at    10   10
\put{$w_4$} at     10  14
\put{$\bullet$} at    20    0
\put{$p_1$} at     21   4
\put{$\bullet$} at    30    0
\put{$p_2$} at     30   4
\put{$\bullet$} at    40    0
\put{$p_3$} at     40   4
\put{$\bullet$} at    50    0
\put{$p_4$} at     48   4
\put{$\bullet$} at    60   10
\put{$t_3$} at     60  15
\put{$\bullet$} at    60  -10
\put{$t_4$} at     60 -15
\plot  -20 0  -10 10  10 10  20 0  10 -10  -10 -10  -20 0 /
\plot  -10 -10  -10 10  10 -10  10 10  -10 -10 /
\plot  -20 0  20 0 /
\plot 20 0   50 0  60 10  60 -10  50 0 /
\setquadratic
\plotsymbolspacing=0.1pt
\plot -10  10   15  10.5  30   0 /  
\plot -10  10   20  11    40   0 /  
\plot  10  10   20   5    30   0 /  
\plot  10  10   25   6    40   0 /  
\plot -10 -10   34 -12.5  60  10 /  
\plot -10 -10   20 -16    60 -10 /  
\plot  10 -10   40  -8    60  10 /  
\plot  10 -10   30 -12    60 -10 /  
\setlinear
\put{$F_5$} at    -22  12
\endpicture} at   -45  0
\put{\beginpicture
\setcoordinatesystem units <\x,\y>
\setplotarea x from -12 to 12, y from -12 to 12
\put{$\bullet$} at     0    0
\circulararc  360  degrees from   0  2 center at  0   0
\put{$x$} at     0  4
\put{$\bullet$} at   -20    0
\put{$t_1$} at    -21   4
\put{$\bullet$} at   -10   10
\circulararc  360  degrees from  -10  12 center at  -10   10
\put{$w_1$} at    -10  14
\put{$\bullet$} at   -10  -10
\circulararc  360  degrees from  -10 -12 center at  -10  -10
\put{$w_2$} at    -10 -14
\put{$\bullet$} at    10  -10
\circulararc  360  degrees from   10 -12 center at   10  -10
\put{$w_3$} at     10 -14
\put{$\bullet$} at    10   10
\circulararc  360  degrees from   10  12 center at   10   10
\put{$w_4$} at     10  14
\put{$\bullet$} at    20    0
\put{$p_1$} at     21   4
\put{$\bullet$} at    30    0
\put{$p_2$} at     30   4
\put{$\bullet$} at    40    0
\put{$p_3$} at     40   4
\put{$\bullet$} at    50    0
\put{$p_4$} at     48   4
\put{$\bullet$} at    60   10
\put{$t_3$} at     60  15
\put{$\bullet$} at    60  -10
\put{$t_4$} at     60 -15
\plot  -20 0  -10 10  10 10  20 0  10 -10  -10 -10  -20 0 /
\plot  -10 -10  -10 10  10 -10  10 10  -10 -10 /
\plot  -20 0  20 0 /
\plot 20 0   50 0  60 10  60 -10  50 0 /
\setquadratic
\plotsymbolspacing=0.1pt
\plot -10  10   20  11    40   0 /  
\plot -10  10   20  12    50   0 /  
\plot  10  10   25   6    40   0 /  
\plot  10  10   25   7    50   0 /  
\plot -10 -10   34 -12.5  60  10 /  
\plot -10 -10   20 -16    60 -10 /  
\plot  10 -10   40  -8    60  10 /  
\plot  10 -10   30 -12    60 -10 /  
\setlinear
\put{$F_6$} at    -22  12
\endpicture} at    0  0
\put{\beginpicture
\setcoordinatesystem units <\x,\y>
\setplotarea x from -12 to 12, y from -12 to 12
\put{$\bullet$} at     0    0
\put{$\bullet$} at     0    0
\circulararc  360  degrees from   0  2 center at  0   0
\put{$x$} at     0  4
\put{$\bullet$} at   -20    0
\put{$t_1$} at    -21   4
\put{$\bullet$} at   -10   10
\circulararc  360  degrees from  -10  12 center at  -10   10
\put{$w_1$} at    -10  14
\put{$\bullet$} at   -10  -10
\circulararc  360  degrees from  -10 -12 center at  -10  -10
\put{$w_2$} at    -10 -14
\put{$\bullet$} at    10  -10
\circulararc  360  degrees from   10 -12 center at   10  -10
\put{$w_3$} at     10 -14
\put{$\bullet$} at    10   10
\circulararc  360  degrees from   10  12 center at   10   10
\put{$w_4$} at     10  14
\put{$\bullet$} at    20    0
\put{$p_1$} at     21   4
\put{$\bullet$} at    30    0
\put{$p_2$} at     30   4
\put{$\bullet$} at    40    0
\put{$p_3$} at     40   4
\put{$\bullet$} at    50    0
\put{$p_4$} at     48   4
\put{$\bullet$} at    60   10
\put{$t_3$} at     60  15
\put{$\bullet$} at    60  -10
\put{$t_4$} at     60 -15
\plot  -20 0  -10 10  10 10  20 0  10 -10  -10 -10  -20 0 /
\plot  -10 -10  -10 10  10 -10  10 10  -10 -10 /
\plot  -20 0  20 0 /
\plot 20 0   50 0  60 10  60 -10  50 0 /
\setquadratic
\plotsymbolspacing=0.1pt
\plot -10  10   20  12    50   0 /  
\plot -10  10   34  12.5  60 -10 /  
\plot -10  10   20  16    60  10 /  
\plot  10  10   25   7    50   0 /  
\plot  10  10   40   8    60 -10 /  
\plot  10  10   30  12    60  10 /  
\plot -10 -10   20 -16    60 -10 /  
\plot  10 -10   30 -12    60 -10 /  
\setlinear
\put{$F_7$} at    -22  12
\endpicture} at   45   0
\put{\beginpicture
\setcoordinatesystem units <\x,\y>
\setplotarea x from -12 to 12, y from -12 to 12
\put{$\bullet$} at     0    0
\circulararc  360  degrees from   0  2 center at  0   0
\put{$x$} at     0  4
\put{$\bullet$} at   -20    0
\put{$t_1$} at    -21   4
\put{$\bullet$} at   -10   10
\circulararc  360  degrees from  -10  12 center at  -10   10
\put{$w_1$} at    -10  14
\put{$\bullet$} at   -10  -10
\circulararc  360  degrees from  -10 -12 center at  -10  -10
\put{$w_2$} at    -10 -14
\put{$\bullet$} at    10  -10
\circulararc  360  degrees from   10 -12 center at   10  -10
\put{$w_3$} at     10 -14
\put{$\bullet$} at    10   10
\circulararc  360  degrees from   10  12 center at   10   10
\put{$w_4$} at     10  14
\put{$\bullet$} at    20    0
\put{$p_1$} at     21   4
\put{$\bullet$} at    30    0
\put{$p_2$} at     30   4
\put{$\bullet$} at    40    0
\put{$p_3$} at     40   4
\put{$\bullet$} at    50    0
\put{$p_4$} at     48   4
\put{$\bullet$} at    60   10
\put{$t_3$} at     60  15
\put{$\bullet$} at    60  -10
\put{$t_4$} at     60 -15
\plot  -20 0  -10 10  10 10  20 0  10 -10  -10 -10  -20 0 /
\plot  -10 -10  -10 10  10 -10  10 10  -10 -10 /
\plot  -20 0  20 0 /
\plot 20 0   50 0  60 10  60 -10  50 0 /
\setquadratic
\plotsymbolspacing=0.1pt
\plot -10  10   15  10.5  30   0 /  
\plot -10  10   20  11    40   0 /  
\plot  10  10   20   5    30   0 /  
\plot  10  10   25   6    40   0 /  
\plot -10 -10   20 -12    50   0 /  
\plot -10 -10   34 -12.5  60  10 /  
\plot -10 -10   20 -16    60 -10 /  
\plot  10 -10   25  -7    50   0 /  
\plot  10 -10   40  -8    60  10 /  
\plot  10 -10   30 -12    60 -10 /  
\setlinear
\put{$F_8$} at    -22  12
\endpicture} at   -45  -30
\put{\beginpicture
\setcoordinatesystem units <\x,\y>
\setplotarea x from -12 to 12, y from -12 to 12
\put{$\bullet$} at     0    0
\put{$\bullet$} at     0    0
\circulararc  360  degrees from   0  2 center at  0   0
\put{$x$} at     0  4
\put{$\bullet$} at   -20    0
\put{$t_1$} at    -21   4
\put{$\bullet$} at   -10   10
\circulararc  360  degrees from  -10  12 center at  -10   10
\put{$w_1$} at    -10  14
\put{$\bullet$} at   -10  -10
\circulararc  360  degrees from  -10 -12 center at  -10  -10
\put{$w_2$} at    -10 -14
\put{$\bullet$} at    10  -10
\circulararc  360  degrees from   10 -12 center at   10  -10
\put{$w_3$} at     10 -14
\put{$\bullet$} at    10   10
\circulararc  360  degrees from   10  12 center at   10   10
\put{$w_4$} at     10  14
\put{$\bullet$} at    20    0
\put{$p_1$} at     21   4
\put{$\bullet$} at    30    0
\put{$p_2$} at     30   4
\put{$\bullet$} at    40    0
\put{$p_3$} at     40   4
\put{$\bullet$} at    50    0
\put{$p_4$} at     48   4
\put{$\bullet$} at    60   10
\put{$t_3$} at     60  15
\put{$\bullet$} at    60  -10
\put{$t_4$} at     60 -15
\plot  -20 0  -10 10  10 10  20 0  10 -10  -10 -10  -20 0 /
\plot  -10 -10  -10 10  10 -10  10 10  -10 -10 /
\plot  -20 0  20 0 /
\plot 20 0   50 0  60 10  60 -10  50 0 /
\setquadratic
\plotsymbolspacing=0.1pt
\plot -10  10   20  11    40   0 /  
\plot -10  10   20  12    50   0 /  
\plot  10  10   25   6    40   0 /  
\plot  10  10   25   7    50   0 /  
\plot -10 -10   20 -12    50   0 /  
\plot -10 -10   34 -12.5  60  10 /  
\plot -10 -10   20 -16    60 -10 /  
\plot  10 -10   25  -7    50   0 /  
\plot  10 -10   40  -8    60  10 /  
\plot  10 -10   30 -12    60 -10 /  
\setlinear
\put{$F_9$} at    -22  12
\endpicture} at    0  -30
\put{\beginpicture
\setcoordinatesystem units <\x,\y>
\setplotarea x from -12 to 12, y from -12 to 12
\put{$\bullet$} at     0    0
\put{$\bullet$} at     0    0
\circulararc  360  degrees from   0  2 center at  0   0
\put{$x$} at     0  4
\put{$\bullet$} at   -20    0
\put{$t_1$} at    -21   4
\put{$\bullet$} at   -10   10
\circulararc  360  degrees from  -10  12 center at  -10   10
\put{$w_1$} at    -10  14
\put{$\bullet$} at   -10  -10
\circulararc  360  degrees from  -10 -12 center at  -10  -10
\put{$w_2$} at    -10 -14
\put{$\bullet$} at    10  -10
\circulararc  360  degrees from   10 -12 center at   10  -10
\put{$w_3$} at     10 -14
\put{$\bullet$} at    10   10
\circulararc  360  degrees from   10  12 center at   10   10
\put{$w_4$} at     10  14
\put{$\bullet$} at    20    0
\put{$p_1$} at     21   4
\put{$\bullet$} at    30    0
\put{$p_2$} at     30   4
\put{$\bullet$} at    40    0
\put{$p_3$} at     40   4
\put{$\bullet$} at    50    0
\put{$p_4$} at     48   4
\put{$\bullet$} at    60   10
\put{$t_3$} at     60  15
\put{$\bullet$} at    60  -10
\put{$t_4$} at     60 -15
\plot  -20 0  -10 10  10 10  20 0  10 -10  -10 -10  -20 0 /
\plot  -10 -10  -10 10  10 -10  10 10  -10 -10 /
\plot  -20 0  20 0 /
\plot 20 0   50 0  60 10  60 -10  50 0 /
\setquadratic
\plotsymbolspacing=0.1pt
\plot -10  10   20  12    50   0 /  
\plot -10  10   35  12.5  60 -10 /  
\plot -10  10   20  16    60  10 /  
\plot  10  10   25   7    50   0 /  
\plot  10  10   41   8    60 -10 /  
\plot  10  10   30  12    60  10 /  
\plot -10 -10   35 -12.5  60  10 /  
\plot -10 -10   20 -16    60 -10 /  
\plot  10 -10   41  -8    60  10 /  
\plot  10 -10   30 -12    60 -10 /  
\setlinear
\put{$F_{10}$} at    -22  12
\endpicture} at   45  -30
\ep$$
\bsm
\caption{The 10 exceptional graphs}
\label{fig-W4-10_exceptions}
\end{figure}

\vspace*{-1mm}
  
  Each of the graphs $F_1,\ldots,F_{10}$ is $\{\claw,\Gt\}$-free, we have
  $F_i\indsub\bG$ for some $i\in\{1,\ldots,10\}$, however, in $(F_i)^{^*}_x$ 
  we have $\lab x,t_1,p_1,p_2,p_3,p_4,t_3,t_4\rab_{(F_i)^{^*}_x}\iso\Gt$.
  By the definition of the $\Gt$-closure, in each of these possibilities 
  not only $(F_i)^{^*}_x$, but even $\Gst_M$ for any (weakly) feasible set 
  $M\subset V(\bG)$, contains an induced $\Gt$. To reach a contradiction, 
  in each of the subgraphs $F_1,\ldots,F_{10}$, we identify a weakly feasible set 
  $M\subset V(\bG)$ with $|M|>1$ and $x\in M$, for which this is not possible.

  First observe that, in all the graphs $F_1,\ldots,F_{10}$, $x$ is locally 
  is 2-connected in $\bG$, hence $x$ is feasible in $\bG$ by 
  Theorem~\ref{thmA-HC-2-conn_neighb}. Set $G_1=(\bG)^{^*}_x$.

\vspace*{-1mm}

  %
  %
  \setcounter{prostrclaim}{0}
  \begin{claim}
  \label{claim-W4-1}
  Let $v\in\{w_1,w_2,w_3,w_4\}$, and let $S$ be the graph in 
  Fig.~\ref{fig-graph_S}. Then either $v$ is locally 2-connected in $G_1$, or 
  $v$ is not a vertex of degree~4 in an induced subgraph $F\indsub G_1$ such 
  that $F\iso S$.
  \end{claim}
    
  \bsm
     
  \begin{proofcl}
  Let, say, $v=w_1$. Then $w_1$ and $w_4$ have, in all the graphs 
  $F_1,\ldots,F_{10}$, a common neighbor $u\in\{p_3,p_4,t_3,t_4\}$. Since $\bG$ 
  is 3-connected and not Hamilton-connected, by Theorem~\ref{thmA-Chvatal-Erdos},
  $\alpha(\bG)\geq 3$. By Proposition~\ref{prop-closure-W5}, $\bG$ is $W_5$-free,
  thus, by Theorem~\ref{thmA-Fouquet}, $\NbG(v)$, hence also $\NbG[v]$, can be
  covered by two cliques, say, $K_1$ and $K_2$. 
  Choose the notation such that $K_1$ contains the triangle $w_1w_4t_1$, and 
  $K_2$ contains the triangle $w_1w_4u$. Then, in $G_1$, $K_1$ extends to 
  a clique $K_1'$ containing the vertices $w_1,w_2,w_3,w_4,t_1,x,p_1$. Thus, 
  $\{w_1,w_4\}\subset K_1'\cap K_2$. 
  
  If $|K_1'\cap K_2|\geq 3$, then $w_1$ is locally 2-connected in $G_1$ and we 
  are done by Theorem~\ref{thmA-HC-2-conn_neighb}; thus, let 
  $K_1'\cap K_2=\{w_1,w_4\}$.
  Then, if $w_1$ is a vertex of degree~4 in $F\indsub G_1$ with $F\iso S$, there 
  are vertices $z_i\in K_i\sm\{w_1,w_4\}$, $i=1,2$, such that $z_1z_2\in E(G_1)$,
  but then $w_1$ is locally 2-connected in $G_1$.
  
  The proof for $v\in\{w_2,w_3,w_4\}$ is symmetric (since the argument in our proof 
  never used the vertex $p_2$).
  \end{proofcl}

\bsm\bsm

  %
  %
  \begin{claim}
  \label{claim-W4-2}
  At least one of the vertices $w_1,w_2,w_3,w_4$ is feasible in $G_1$.
  \end{claim}
    
  \bsm\bsm
     
  \begin{proofcl}
  If, say, $w_1$ is not feasible in $G_1$, then, by Claim~\ref{claim-W4-1} and by
  Proposition~\ref{propA-subgraph_S}$(ii)$, $(G_1)^{^*}_{w_1}$ has a hamiltonian
  $(w_1,w)$-path for some $w\in N_{G_1}(w_1)$. Hence there is no hamiltonian
  $(w_1,w)$-path in $G_1$, and, by Proposition~\ref{propA-subgraph_S}$(ii)$,
  there is still no hamiltonian $(w_1,w)$-path in $(G_1)^{^*}_{w'}$ for any 
  $w'\in\{w_2,w_3,w_4\}\sm\{w\}$. Thus, $w'$ is feasible in $G_1$.
  \end{proofcl}

\bsm\bsm

  %
  %
  \begin{claim}
  \label{claim-W4-3}
  Let $F_i\indsub\bG$ for some $i$, $1\leq i\leq 10$, and set 
  $M_i=\{x,w_1,w_2,w_3,w_4\}$ if $i\in\{1,3,4,6,7,8,9,10\}$, or
  $M_2=\{x,w_2,w_3\}$, or $M_5=\{x,w_1,w_4\}$.
  Then the set $M_i$ is weakly feasible in $\bG$.
  \end{claim}

\bsm

  \begin{proofcl}
  Throughout this proof, we will use $\VLDC(G)$ to denote the set of all
  locally 2-connected vertices in a graph $G$.

  First of all, observe that $x\in\VLDC(\bG)$ in all cases
  $i=1,\ldots,10$ by Lemma~\ref{lemmaA-2_indep_sets} (the independent sets are 
  $\{w_1,w_3\}$ and $\{w_2,w_4\}$), hence $x$ is feasible in $\bG$ by 
  Theorem~\ref{thmA-HC-2-conn_neighb}, and feasible or simplicial in any graph 
  obtained from $\bG$ by a series of local completions.
 
  Let $F_j\indsub\bG$, and let, by Claim~\ref{claim-W4-2},  
  $v_j\in\{w_1,w_2,w_3,w_4\}\cap\VLDC(G_1)$.
  
  \begin{mylist}
    \item[\underline{Case Cl~\ref{claim-W4-3}-1:}] {\sl $F_1\indsub\bG$.} \quad \\
    By symmetry, we can assume that $v_1=w_4$. 
    Then, by Lemma~\ref{lemmaA-2_indep_sets}, $w_3\in\VLDC(\bGst_{w_4})$ (the
    independent sets are $\{w_2,w_4\}$ and $\{w_1,t_4\}$), and also 
    $w_1\in\VLDC(\bGst_{w_4})$ (the independent sets are $\{w_2,w_4\}$ and 
    $\{t_1,t_3\}$). By symmetry, $w_2\in\VLDC(\bGst_{w_1})$.
    Hence also $w_3\in\VLDC((\bGst_{w_4})^{^*}_x)=\VLDC((\bGst_{x})^{^*}_{w_4})$,
    i.e., $w_3$ is feasible or simplicial in $(\bGst_{x})^{^*}_{w_4}$.
    Similarly with $w_1$ and $w_2$. 

    Thus, if $F_1\indsub\bG$, then the set  $M_1=\{x,w_1,w_2,w_3,w_4\}$
    is weakly feasible in $\bG$.
    
    We summarize the above discussion in the following table (where ``L2C"
    stands for ``locally 2-connected").
    
   \begin{tabular}{|c|l|l|}
   \hline
   $v_1$  &  L2C vertex & Argument \\
   \hline
   $w_4$  & $w_3\in\VLDC(\bGst_{w_4})$ & Lemma~\ref{lemmaA-2_indep_sets},
                                  indep. sets $\{w_2,w_4\}$, $\{x,t_4\}$ \\
          & $w_1\in\VLDC(\bGst_{w_4})$ & Lemma~\ref{lemmaA-2_indep_sets},
                                  indep. sets $\{w_2,w_4\}$, $\{t_1,t_3\}$ \\      
          & $w_2\in\VLDC(\bGst_{w_1})$ & Symmetric to 
                                  $w_3\in\VLDC(\bGst_{w_4})$ \\
   \hline
   \end{tabular}

   \item[\underline{Case Cl~\ref{claim-W4-3}-2:}] {\sl $F_2\indsub\bG$.} \quad \\
   Immediately $w_3\in\VLDC(\bG)$ (independent sets $\{w_2,w_4\}$ and 
   $\{x,p_3\}$), and then $w_2\in\VLDC(\bGst_{w_3})$ (independent sets 
   $\{w_1,w_3\}$ and $\{t_1,p_3\})$. 

   Thus, if $F_2\indsub\bG$, then the set  $M_2=\{x,w_2,w_3\}$
   is weakly feasible in $\bG$.

   \item[\underline{Case Cl~\ref{claim-W4-3}-3:}] {\sl $F_3\indsub\bG$.} \quad \\
   By symmetry, we can assume that $v_3\in\{w_3,w_4\}$. 
   We summarize the possibilities in the following table.

   \begin{tabular}{|c|l|l|}
   \hline
   $v_3$  &  L2C vertex & Argument \\
   \hline
   $w_4$  & $w_3\in\VLDC(\bGst_{w_4})$ & Lemma~\ref{lemmaA-2_indep_sets},
                                  indep. sets $\{w_2,w_4\}$, $\{x,p_4\}$ \\
          & $w_1\in\VLDC(\bGst_{w_4})$ & Lemma~\ref{lemmaA-2_indep_sets},
                                  indep. sets $\{w_2,w_4\}$, $\{t_1,t_3\}$ \\      
          & $w_2\in\VLDC(\bGst_{w_1})$ & Symmetric to 
                                  $w_3\in\VLDC(\bGst_{w_4})$ \\
   \hline          
   $w_3$  & $w_4\in\VLDC(\bGst_{w_3})$ & Lemma~\ref{lemmaA-2_indep_sets},
                                  indep. sets $\{w_1,w_3\}$, $\{x,t_3\}$ \\
          & $w_2\in\VLDC(\bGst_{w_3})$ & Lemma~\ref{lemmaA-2_indep_sets},
                                  indep. sets $\{w_1,w_3\}$, $\{t_1,p_4\}$ \\      
          & $w_1\in\VLDC(\bGst_{w_2})$ & Symmetric to $w_4\in\VLDC(\bGst_{w_3})$ \\          
   \hline
   \end{tabular}

    Thus, if $F_3\indsub\bG$, then the set  $M_3=\{x,w_1,w_2,w_3,w_4\}$
    is weakly feasible in $\bG$.

   \item[\underline{Case Cl~\ref{claim-W4-3}-4:}] {\sl $F_4\indsub\bG$.} \quad \\
   In this case, already $w_i\in\VLDC(\bG)$, $i=1,2,3,4$:

   \begin{tabular}{|l|l|}
   \hline
   L2C vertex & Argument \\
   \hline
   $w_1\in\VLDC(\bG)$ & Lemma~\ref{lemmaA-2_indep_sets},
                                  indep. sets $\{w_2,w_4\}$, $\{x,p_3\}$ \\
   $w_2\in\VLDC(\bG)$ & Lemma~\ref{lemmaA-2_indep_sets},
                                  indep. sets $\{w_1,w_3\}$, $\{x,p_3\}$ \\      
   $w_3\in\VLDC(\bG)$ & Lemma~\ref{lemmaA-2_indep_sets},
                                  indep. sets $\{w_2,w_4\}$, $\{x,p_3\}$ \\
   $w_4\in\VLDC(\bG)$ & Lemma~\ref{lemmaA-2_indep_sets},
                                  indep. sets $\{w_1,w_3\}$, $\{x,p_3\}$ \\      
   \hline
   \end{tabular}

   Thus, if $F_4\indsub\bG$, then the set  $M_4=\{x,w_1,w_2,w_3,w_4\}$
   is weakly feasible in $\bG$.

   \item[\underline{Case Cl~\ref{claim-W4-3}-5:}] {\sl $F_5\indsub\bG$.} \quad \\
   In this case, $w_4\in\VLDC(\bG)$ (Lemma~\ref{lemmaA-2_indep_sets} with sets 
   $\{w_1,w_3\}$, $\{x,p_3\}$), and then 
   $w_1\in\VLDC(\bGst_{w_4})$ (Lemma~\ref{lemmaA-2_indep_sets} with sets 
   $\{w_2,w_4\}$, $\{t_1,p_3\}$).   

   Thus, if $F_5\indsub\bG$, then the set  $M_5=\{x,w_1,w_4\}$ is weakly 
   feasible in $\bG$.

   \item[\underline{Case Cl~\ref{claim-W4-3}-6:}] {\sl $F_6\indsub\bG$.} \quad \\
   By symmetry, we can assume that $v_6\in\{w_3,w_4\}$. We then have the following 
   possibilities.
   
   \begin{tabular}{|c|l|l|}
   \hline
   $v_6$  &  L2C vertex & Argument \\
   \hline
   $w_3$  & $w_2\in\VLDC(\bGst_{w_3})$ & Lemma~\ref{lemmaA-2_indep_sets},
                                  indep. sets $\{w_1,w_3\}$, $\{t_1,t_3\}$ \\
          & $w_4\in\VLDC(\bGst_{w_3})$ & Lemma~\ref{lemmaA-2_indep_sets},
                                  indep. sets $\{w_1,w_3\}$, $\{x,p_3\}$ \\      
          & $w_1\in\VLDC(\bGst_{w_2})$ & Symmetric to 
                                  $w_4\in\VLDC(\bGst_{w_3})$ \\
   \hline          
   $w_4$  & $w_1\in\VLDC(\bGst_{w_4})$ & Lemma~\ref{lemmaA-2_indep_sets},
                                  indep. sets $\{w_2,w_4\}$, $\{t_1,p_4\}$ \\
          & $w_3\in\VLDC(\bGst_{w_4})$ & Lemma~\ref{lemmaA-2_indep_sets},
                                  indep. sets $\{w_2,w_4\}$, $\{x,t_3\}$ \\      
          & $w_2\in\VLDC(\bGst_{w_1})$ & Symmetric to $w_3\in\VLDC(\bGst_{w_4})$ \\          
   \hline
   \end{tabular}

   Thus, if $F_6\indsub\bG$, then the set  $M_6=\{x,w_1,w_2,w_3,w_4\}$
   is weakly feasible in $\bG$.

   \item[\underline{Case Cl~\ref{claim-W4-3}-7:}] {\sl $F_7\indsub\bG$.} \quad \\
   By symmetry, it is sufficient to verify $w_3$ and $w_4$, however, 
   $w_3,w_4\in\VLDC(\bG)$ by Lemma~\ref{lemmaA-2_indep_sets}: $w_3$ with sets 
   $\{w_2,w_4\}$, $\{x,t_4\}$, $w_4$ with sets $\{w_1,w_3\}$, $\{x,p_4\}$.

   Thus, if $F_7\indsub\bG$, then the set  $M_7=\{x,w_1,w_2,w_3,w_4\}$
   is weakly feasible in $\bG$.

   \item[\underline{Case Cl~\ref{claim-W4-3}-8:}] {\sl $F_8\indsub\bG$.} \quad \\
   Then $w_4\in\VLDC(\bG)$ (Lemma~\ref{lemmaA-2_indep_sets} with sets 
   $\{w_2,w_4\}$, $\{x,p_3\}$), and for the remaining vertices we have the 
   following possibilities.

   \begin{tabular}{|l|l|}
   \hline
   L2C vertex & Argument \\
   \hline
   $w_1\in\VLDC(\bGst_{w_4})$ & Lemma~\ref{lemmaA-2_indep_sets},
                                  indep. sets $\{w_2,w_4\}$, $\{t_1,p_3\}$ \\
   $w_3\in\VLDC(\bGst_{w_4})$ & Lemma~\ref{lemmaA-2_indep_sets},
                                  indep. sets $\{w_2,w_4\}$, $\{x,t_3\}$ \\      
   $w_2\in\VLDC(\bGst_{w_1})$ & Symmetric to $w_3\in\VLDC(\bGst_{w_4})$  \\
   \hline
   \end{tabular}

   Thus, if $F_8\indsub\bG$, then the set  $M_8=\{x,w_1,w_2,w_3,w_4\}$
   is weakly feasible in $\bG$.
 
   \item[\underline{Case Cl~\ref{claim-W4-3}-9:}] {\sl $F_9\indsub\bG$.} \quad \\
   In this case, $\{w_1,w_2,w_3,w_4\}\subset\VLDC(\bG)$: $w_1$ and $w_2$ by 
   Lemma~\ref{lemmaA-2_indep_sets} with sets $\{w_2,w_4\}$, $\{t_1,p_3\}$) 
   for $w_1$, and $\{w_1,w_3\}$, $\{x,p_4\}$) for $w_2$; $w_3$ and $w_4$ follow
   by symmetry.

   Thus, if $F_9\indsub\bG$, then the set  $M_9=\{x,w_1,w_2,w_3,w_4\}$
   is weakly feasible in $\bG$.

   \item[\underline{Case Cl~\ref{claim-W4-3}-10:}] {\sl $F_{10}\indsub\bG$.} \quad \\
   In this case similarly $\{w_1,w_2,w_3,w_4\}\subset\VLDC(\bG)$: $w_1$ and $w_2$ 
   by Lemma~\ref{lemmaA-2_indep_sets} with sets $\{w_2,w_4\}$, $\{x,p_4\}$) 
   for $w_1$, and $\{w_1,w_3\}$, $\{x,t_3\}$) for $w_2$; $w_3$ and $w_4$ follow
   by symmetry.

   Thus, if $F_{10}\indsub\bG$, then the set  $M_{10}=\{x,w_1,w_2,w_3,w_4\}$
   is weakly feasible in $\bG$.
  \end{mylist}
  \vspace*{-9mm}
  \end{proofcl}

\bsm

  Now, since each of the sets $M_i$ is weakly feasible in $\bG$, by the 
  definition of the $\Gt$-closure, each $\bGst_{M_i}$ contains an induced 
  subgraph $F'\iso\Gt$, $i=1,\ldots,10$. Since each $N_{\bG}[M_i]$ induces in 
  $\bGst_{M_i}$ a clique, the clique $\la N_{\bG}[M_i]\ra_{\bGst_{M_i}}$
  contains either a triangle of $F'$, or one of the edges of the path of $F'$.
  The rest of $F'$ outside $\la N_{\bG}[M_i]\ra_{\bGst_{M_i}}$ consists of 
  a triangle and a path of appropriate length in the first case, or of two
  triangles with a path of appropriate length in the second case.
  Let us call these parts of $F'$ outside $\la N_{\bG}[M_i]\ra_{\bGst_{M_i}}$
  ''tails".
  Then each tail consists of a triangle with a path, and the length of the path 
  is 3 in the first case, or the lengths of the paths sum up to 2 in the second 
  case. Moreover, since each vertex of $M_i$ is simplicial in $\bGst_{M_i}$, it 
  cannot be an end of a tail directly, but a tail can be attached to some its 
  neighbor, which corresponds to the situation with a tail one longer.

  Considering all possible combinations of such tails, we obtain the possibilities 
  that are listed in the following table.

  \ssk

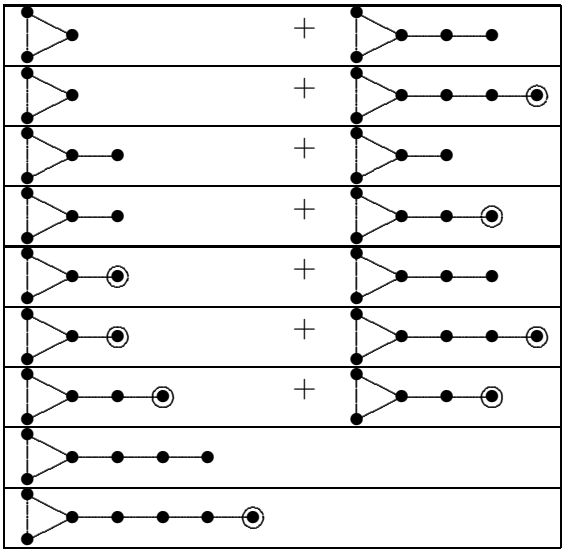
  \quad \begin{tabular}{|lll|}
  \hline
  $$\beginpicture
  \put {\bp
  \setcoordinatesystem units <0.6mm,0.6mm>
    \put{$\bullet$} at  -25  -5
    \put{$\bullet$} at  -25   5
    \put{$\bullet$} at  -15   0
    \plot -15 0  -25 -5  -25 5 -15 0 /
    \ep} at 0 0
  \endpicture$$ 
  &
  $+$
  &
  $$\beginpicture
  \put {\bp
  \setcoordinatesystem units <0.6mm,0.6mm>
    \put{$\bullet$} at  -25  -5
    \put{$\bullet$} at  -25   5
    \put{$\bullet$} at  -15   0
    \put{$\bullet$} at   -5   0
    \put{$\bullet$} at    5   0
    \plot -15 0  -25 -5  -25 5 -15 0  5 0  /
  \ep} at 0 0
  \endpicture$$  \\
  \hline
  $$\beginpicture
  \put {\bp
  \setcoordinatesystem units <0.6mm,0.6mm>
    \put{$\bullet$} at  -25  -5
    \put{$\bullet$} at  -25   5
    \put{$\bullet$} at  -15   0
    \plot -15 0  -25 -5  -25 5 -15 0   /
  \ep} at 0 0
  \endpicture$$ 
  &
  $+$
  &
  $$\beginpicture
  \put {\bp
  \setcoordinatesystem units <0.6mm,0.6mm>
    \put{$\bullet$} at  -25  -5
    \put{$\bullet$} at  -25   5
    \put{$\bullet$} at  -15   0
    \put{$\bullet$} at   -5   0
    \put{$\bullet$} at    5   0
    \put{$\bullet$} at   15   0
    \circulararc  360  degrees from  17.3  0 center at  15   0
    \plot -15 0  -25 -5  -25 5 -15 0  15 0  /
  \ep} at 0 0
  \endpicture$$  \\
  \hline
  $$\beginpicture
  \put {\bp
  \setcoordinatesystem units <0.6mm,0.6mm>
    \put{$\bullet$} at  -25  -5
    \put{$\bullet$} at  -25   5
    \put{$\bullet$} at  -15   0
    \put{$\bullet$} at   -5   0
    \plot -15 0  -25 -5  -25 5 -15 0  -5 0  /
  \ep} at 0 0
  \endpicture$$ 
  &
  $+$
  &
  $$\beginpicture
  \put {\bp
  \setcoordinatesystem units <0.6mm,0.6mm>
    \put{$\bullet$} at  -25  -5
    \put{$\bullet$} at  -25   5
    \put{$\bullet$} at  -15   0
    \put{$\bullet$} at   -5   0
    \plot -15 0  -25 -5  -25 5 -15 0  -5 0  /
  \ep} at 0 0
  \endpicture$$  \\
  \hline
  $$\beginpicture
  \put {\bp
  \setcoordinatesystem units <0.6mm,0.6mm>
    \put{$\bullet$} at  -25  -5
    \put{$\bullet$} at  -25   5
    \put{$\bullet$} at  -15   0
    \put{$\bullet$} at   -5   0
    \plot -15 0  -25 -5  -25 5 -15 0  -5 0  /
  \ep} at 0 0
  \endpicture$$ 
  &
  $+$
  &
  $$\beginpicture
  \put {\bp
  \setcoordinatesystem units <0.6mm,0.6mm>
    \put{$\bullet$} at  -25  -5
    \put{$\bullet$} at  -25   5
    \put{$\bullet$} at  -15   0
    \put{$\bullet$} at   -5   0
    \put{$\bullet$} at    5   0
    \circulararc  360  degrees from   7.3  0 center at   5   0
    \plot -15 0  -25 -5  -25 5 -15 0  5 0  /
  \ep} at 0 0
  \endpicture$$  \\
  \hline
  $$\beginpicture
  \put {\bp
  \setcoordinatesystem units <0.6mm,0.6mm>
    \put{$\bullet$} at  -25  -5
    \put{$\bullet$} at  -25   5
    \put{$\bullet$} at  -15   0
    \put{$\bullet$} at   -5   0
    \circulararc  360  degrees from  -7.3  0 center at  -5   0
    \plot -15 0  -25 -5  -25 5 -15 0  -5 0  /
  \ep} at 0 0
  \endpicture$$ 
  &
  $+$
  &
  $$\beginpicture
  \put {\bp
  \setcoordinatesystem units <0.6mm,0.6mm>
    \put{$\bullet$} at  -25  -5
    \put{$\bullet$} at  -25   5
    \put{$\bullet$} at  -15   0
    \put{$\bullet$} at   -5   0
    \put{$\bullet$} at   5   0
    \plot -15 0  -25 -5  -25 5  -15 0  5 0  /
  \ep} at 0 0
  \endpicture$$  \\
  \hline
  $$\beginpicture
  \put {\bp
  \setcoordinatesystem units <0.6mm,0.6mm>
    \put{$\bullet$} at  -25  -5
    \put{$\bullet$} at  -25   5
    \put{$\bullet$} at  -15   0
    \put{$\bullet$} at   -5   0
    \circulararc  360  degrees from  -7.3  0 center at  -5   0
    \plot -15 0  -25 -5  -25 5 -15 0  -5 0  /
  \ep} at 0 0
  \endpicture$$ 
  &
  $+$
  &
  $$\beginpicture
  \put {\bp
  \setcoordinatesystem units <0.6mm,0.6mm>
    \put{$\bullet$} at  -25  -5
    \put{$\bullet$} at  -25   5
    \put{$\bullet$} at  -15   0
    \put{$\bullet$} at   -5   0
    \put{$\bullet$} at    5   0
    \put{$\bullet$} at   15   0
    \circulararc  360  degrees from  17.3  0 center at  15   0
    \plot -15 0  -25 -5  -25 5 -15 0  15 0  /
  \ep} at 0 0
  \endpicture$$  \\
  \hline
  $$\beginpicture
  \put {\bp
  \setcoordinatesystem units <0.6mm,0.6mm>
    \put{$\bullet$} at  -25  -5
    \put{$\bullet$} at  -25   5
    \put{$\bullet$} at  -15   0
    \put{$\bullet$} at   -5   0
    \put{$\bullet$} at    5   0
    \circulararc  360  degrees from   7.3  0 center at   5   0
    \plot -15 0  -25 -5  -25 5 -15 0  5 0  /
  \ep} at 0 0
  \endpicture$$ 
  &
  $+$
  &
  $$\beginpicture
  \put {\bp
  \setcoordinatesystem units <0.6mm,0.6mm>
    \put{$\bullet$} at  -25  -5
    \put{$\bullet$} at  -25   5
    \put{$\bullet$} at  -15   0
    \put{$\bullet$} at   -5   0
    \put{$\bullet$} at    5   0
    \circulararc  360  degrees from   7.3  0 center at   5   0
    \plot -15 0  -25 -5  -25 5 -15 0  5 0  /
  \ep} at 0 0
  \endpicture$$  \\
  \hline
  $$\beginpicture
  \put {\bp
  \setcoordinatesystem units <0.6mm,0.6mm>
    \put{$\bullet$} at  -25  -5
    \put{$\bullet$} at  -25   5
    \put{$\bullet$} at  -15   0
    \put{$\bullet$} at   -5   0
    \put{$\bullet$} at    5   0
    \put{$\bullet$} at   15   0
    \plot -15 0  -25 -5  -25 5 -15 0  15 0  /
  \ep} at 0 0
  \endpicture$$ 
  &

  &
  \\
  \hline
  $$\beginpicture
  \put {\bp
  \setcoordinatesystem units <0.6mm,0.6mm>
    \put{$\bullet$} at  -25  -5
    \put{$\bullet$} at  -25   5
    \put{$\bullet$} at  -15   0
    \put{$\bullet$} at   -5   0
    \put{$\bullet$} at    5   0
    \put{$\bullet$} at   15   0
    \put{$\bullet$} at   25   0
    \circulararc  360  degrees from  27.3  0 center at  25   0
    \plot -15 0  -25 -5  -25 5 -15 0  25 0  /
  \ep} at 0 0
  \endpicture$$ 
  &

  &
   \\
  \hline
  \end{tabular}

  \ssk

  Here, the double-circled endvertices 
  of tails can be identified with a vertex in $M_i$, and the non-double-circled 
  endvertices of tails can be identified with a vertex in $N_{\bG}[M_i]\sm M_i$.

  We have generated by computer all possible ways of extending the graph $F_i$,
  $i=1,\ldots,10$, by joining some of the combinations of the tails in the table
  to vertices in $N_{\bG}[M_i]$.
  To the resulting graphs, we have added all possible sets of edges between the 
  new vertices and $F_i$. Finally, we have tested whether the resulting 
  graphs are $\{\Gt,W_5,\claw\}$-free. In each of the possible cases, we have 
  obtained a contradiction (see \cite{computing}).

  \item[\underline{\bf Case 2:}] {\sl The new edge is on the path of $F$.} \quad \\
  If $|V(W)\cap V(F)|=2$, then, choosing the notation such that 
  $w_1,w_3\in V(F)$, Lemma~\ref{lemma-W5-new-in-path} implies $w_2w_4\in E(G)$,
  a contradiction. Hence $|V(W)\cap V(F)|\leq 1$.
  
  \begin{mylist}

    \item[\underline{\bf Subcase 2.1:}] {\sl $|V(W)\cap V(F)|=1$.} \quad \\
    The proof will depend on the position of the new edge on the path of $F$.
    
    \begin{mylist}

      \item[\underline{\bf Subcase 2.1.1:}] {\sl The edge $p_1p_2$ is new.} \quad \\
      We distinguish two subcases.

      \begin{mylist}

        \item[\underline{\bf Subcase 2.1.1.1:}] {\sl $p_1\in V(W)$.} \quad \\
        Choose the notation such that $p_1=w_1$. 
        Since $\lab x,w_1,p_2,w_3\rabgb\niso\claw$, we have $p_2w_3\in E(\bG)$, 
        and from $\lab x,w_2,w_4,p_2\rabgb\niso\claw$, up to a symmetry,
        $p_2w_4\in E(\bG)$. By Lemma~\ref{lemma-W5-new-in-path}, 
        $w_2p_2\notin E(\bG)$. Since 
        $\lab w_1,w_2,x,p_2\car p_3\car p_4,t_3,t_4\rabgb\niso\Gt$, 
        we have $w_2z\in E(\bG)$ for some $z\in\{p_3,p_4,t_3,t_4\}$. Then
        $w_2t_i\notin E(\bG)$ for otherwise 
        $\lab w_2,t_i,x,z\rabgb\iso\claw$, $i=1,2$, and from 
        $\lab w_1,t_i,w_2,w_4\rabgb\niso\claw$ we have $t_iw_4\in E(\bG)$,
        $i=3,4$. Since $\lab t_1,t_2,w_4,p_2,p_3,p_4,t_3,t_4\rabgb\niso\Gt$, 
        we have $w_4z'\in E(\bG)$ for some $z'\in \{p_3,p_4,t_3,t_4\}$, 
        but in each of these cases, $\lab w_4,t_1,x,z'\rabgb\iso\claw$,
        a contradiction.
        
        \item[\underline{\bf Subcase 2.1.1.2:}] {\sl $p_2\in V(W)$.} \quad \\
        Set $p_2=w_1$. Then similarly $\lab x,p_1,w_1,w_3\rabgb\niso\claw$ 
        implies $p_1w_3\in E(\bG)$, $\lab x,p_1,w_2,w_4\rabgb\niso\claw$, 
        implies, up to a symmetry, $p_1w_4\in E(\bG)$, and 
        Lemma~\ref{lemma-W5-new-in-path} implies $w_2p_1\notin E(\bG)$. 
        From $\lab w_1,w_2,w_4,p_3\rabgb\niso\claw$ we now have 
        $w_4p_3\in E(\bG)$ or $w_2p_3\in E(\bG)$. However, if 
        $w_4p_3\in E(\bG)$, then immediately 
        $\lab t_1,t_2,p_1,w_4,p_3\car p_4,t_3,t_4\rabgb\iso\Gt$ since each 
        of the edges $w_4z$, $z\in \{t_1,t_2,p_4,t_3,t_4\}$, yields an
        induced $\claw$ with center at $w_4$.
        Thus, $w_4p_3\notin E(\bG)$ and  $w_2p_3\in E(\bG)$.
        
        Now $w_2t_i\notin E(\bG)$, for otherwise 
        $\lab w_2,t_i,x,p_3\rabgb\iso\claw$, $i=1,2$. Considering 
        $\lab p_1,w_4,x,w_2,p_3,p_4,t_3,t_4\rabgb\niso\Gt$, we have 
        $w_4z\in E(\bG)$ for some $z\in\{p_4,t_3,t_4\}$, or $w_2z'\in E(\bG)$
        for some $z'\in\{p_4,t_3,t_4\}$. However, in the first case 
        $\lab w_4,z,w_1,p_1\rabgb\iso\claw$, and in the second case 
        $\lab w_2,x,p_3,z'\rabgb\iso\claw$ for $z'\in\{t_3,t_4\}$. 
        Thus, $z'=p_4$, i.e., $w_2p_4\in E(\bG)$. But then 
        $\lab t_1,t_2,p_1,x,w_2,p_4,t_3,t_4\rabgb\iso\Gt$, a contradiction.
      \end{mylist}

      \item[\underline{\bf Subcase 2.1.2:}] {\sl The edge $p_2p_3$ is new.} \quad \\
      By symmetry, we can set $w_1=p_2$. As before, from 
      $\lab x,w_1,p_3,w_3\rabgb\niso\claw$ we have $w_1w_3\in E(\bG)$, from
      $\lab x,w_2,w_4,p_3\rabgb\niso\claw$, up to a symmetry, 
      $p_3w_4\in E(\bG)$, and, by Lemma~\ref{lemma-W5-new-in-path},
      $w_2p_3\notin E(\bG)$.

      %
      %
      \begin{claim}
      \label{claim-W4-4}
      If $y\in V(\bG)$ is such that $\{x,w_1,p_3\}\subset N_{\bG}(y)$, 
      then $N_{\bG}(y)\cap\{t_1,t_2,p_1,p_4\car t_3,t_4\}=\emptyset$.
      \end{claim}
     
      \bsm
      
      \begin{proofcl}
      If $yz\in E(\bG)$ for some $z\in \{t_1,t_2,t_3,t_4\}$, then 
      $\lab y,z,w_1,p_3\rabgb\iso\claw$. If both $yp_1\in E(\bG)$ and 
      $yp_4\in E(\bG)$, then $\lab y,p_1,p_4,x\rabgb\iso\claw$, and if 
      $y$ is adjacent to one of $p_1,p_4$, say, $yp_4\in E(\bG)$, then 
      $\lab t_1,t_2,p_1,w_1,y,p_4,t_3,t_4\rabgb\iso\Gt$.
      \end{proofcl}
      
      \bsm
      
      Now, since $w_4p_1\notin E(\bG)$ by Claim~\ref{claim-W4-4}, from
      $\lab w_1,p_1,w_2,w_4\rabgb\niso\claw$ we have $p_1w_2\in E(\bG)$.
      Since $\lab p_1,w_2,w_1,w_4,p_3,p_4,t_3,t_4\rabgb\niso\Gt$, 
      by Claim~\ref{claim-W4-4}, $w_2z\in E(\bG)$ for some 
      $z\in\{p_4,t_3,t_4\}$, but in each of these cases, 
      $\lab w_2,p_1,x,z\rabgb\iso\claw$, a contradiction.
    \end{mylist}

  \item[\underline{\bf Subcase 2.2:}] {\sl $V(W)\cap V(F) = \emptyset $.} \quad \\
  In this case, the vertex $x$ must be adjacent in $G$ to the two vertices of the
  new edge (and to no other vertex of $F$ since $F$ is induced in $\Gstx$).
  Up to a symmetry, there are two possible subcases, namely, that the edge
  $p_2p_3$ is new, see Fig~\ref{fig-W4-two-subcases}$(a)$, and that the edge
  $p_1p_2$ is new, see Fig~\ref{fig-W4-two-subcases}$(b)$.

  %
  %
  \begin{figure}[ht]
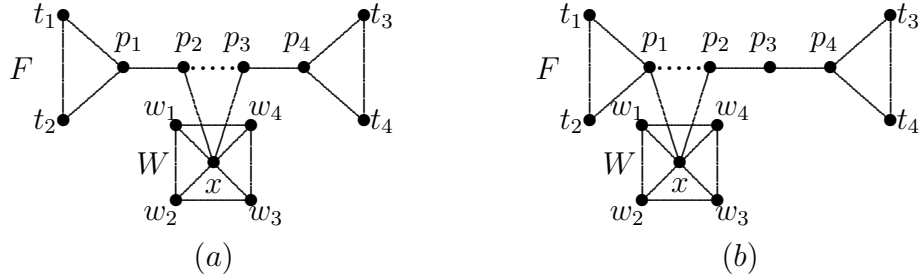

  $$\bp
  \setcoordinatesystem units <1mm,1mm>
  \setplotarea x from -50 to 50, y from -7 to 7
  \put{\beginpicture
  \setcoordinatesystem units <.8mm,.7mm>
  \setplotarea x from -12 to 12, y from -12 to 12
  \put{$\bullet$} at   -15    0
  \put{$p_1$} at    -14   5
  \put{$\bullet$} at    -5    0
  \put{$p_2$} at     -4   5
  \put{$\bullet$} at     5    0
  \put{$p_3$} at      4   5
  \put{$\bullet$} at    15    0
  \put{$p_4$} at     14   5
  \put{$\bullet$} at    -25   10
  \put{$t_1$} at     -28  10
  \put{$\bullet$} at    -25  -10
  \put{$t_2$} at     -28 -10
  \put{$\bullet$} at    25    10
  \put{$t_3$} at      28  10
  \put{$\bullet$} at    25  -10
  \put{$t_4$} at      28 -10
  \plot -15 0  -25 -10  -25 10  -15 0 /
  \plot  15 0   25 -10   25 10   15 0 /
  \setplotsymbol ({\Large .})
  \setdots <3.7pt>
  \plot -5 0  5 0 /
  \setsolid
  \setplotsymbol ({\fiverm .})
  \plot -5 0 -15 0 /
  \plot  5 0  15 0 /
  \put{$F$} at      -32  0
  \put{\beginpicture
  \setcoordinatesystem units <.5mm,.5mm>
  \setplotarea x from -12 to 12, y from -12 to 12
  \put{$\bullet$} at    0   0
  \put{$x$} at    0   -6
  \put{$\bullet$} at   -10   10
  \put{$w_1$} at    -14  14
  \put{$\bullet$} at   -10  -10
  \put{$w_2$} at    -14 -14
  \put{$\bullet$} at    10  -10
  \put{$w_3$} at     14 -14
  \put{$\bullet$} at    10   10
  \put{$w_4$} at     14  14
  \plot -10 -10  -10 10  10 10  10 -10  -10 -10 /
  \plot 10 10  -10 -10  /
  \plot 10 -10  -10 10 /
  \put{$W$} at    -16  0
  \endpicture} at  -0.6  -18
  \plot -5 0    0 -18    5 0 /
  \put{$(a)$} at  0 -36
  \endpicture} at  -35  0
  \put{\beginpicture
  \setcoordinatesystem units <.8mm,.7mm>
  \setplotarea x from -12 to 12, y from -12 to 12
  \put{$\bullet$} at   -15    0
  \put{$p_1$} at    -14   5
  \put{$\bullet$} at    -5    0
  \put{$p_2$} at     -4   5
  \put{$\bullet$} at     5    0
  \put{$p_3$} at      4   5
  \put{$\bullet$} at    15    0
  \put{$p_4$} at     14   5
  \put{$\bullet$} at    -25   10
  \put{$t_1$} at     -28  10
  \put{$\bullet$} at    -25  -10
  \put{$t_2$} at     -28 -10
  \put{$\bullet$} at    25    10
  \put{$t_3$} at      28  10
  \put{$\bullet$} at    25  -10
  \put{$t_4$} at      28 -10
  \plot -15 0  -25 -10  -25 10  -15 0 /
  \plot  15 0   25 -10   25 10   15 0 /
  \setplotsymbol ({\Large .})
  \setdots <3.7pt>
  \plot -15 0  -5 0 /
  \setsolid
  \setplotsymbol ({\fiverm .})
  \plot -5 0  15 0 /
  \put{$F$} at      -32  0
  \put{\beginpicture
  \setcoordinatesystem units <.5mm,.5mm>
  \setplotarea x from -12 to 12, y from -12 to 12
  \put{$\bullet$} at    0   0
  \put{$x$} at    0   -6
  \put{$\bullet$} at   -10   10
  \put{$w_1$} at    -14  14
  \put{$\bullet$} at   -10  -10
  \put{$w_2$} at    -14 -14
  \put{$\bullet$} at    10  -10
  \put{$w_3$} at     14 -14
  \put{$\bullet$} at    10   10
  \put{$w_4$} at     14  14
  \plot -10 -10  -10 10  10 10  10 -10  -10 -10 /
  \plot 10 10  -10 -10  /
  \plot 10 -10  -10 10 /
  \put{$W$} at    -16  0
  \endpicture} at -10.6  -18
  \plot -15 0   -10 -18    -5 0 /
  \put{$(b)$} at  0 -36
  \endpicture} at  35  0
  \ep$$
  \bsm
  \caption{The two possibilities in Subcase 2.2 (dotted lines indicate new edges).}
  \label{fig-W4-two-subcases}
  \end{figure}

  Since the vertex $x$, the vertices of the new edge, and any of the vertices 
  $w_i$, $i=1,2,3,4$, cannot induce a claw in $G$, there must be some more 
  additional edges in $G$ between $\{w_1,w_2,w_3,w_4\}$ and $F$. 
  Checking by computer all possible sets of additional edges $uv$ such that 
  $u\in \{w_1,w_2,w_3,w_4\}$ and $v\in V(F)$, we conclude that, 
  in the first case (when $p_2p_3$ is new), the resulting graph always contains 
  at least one of the graphs $\claw$, $\Gt$ or $W_5$ as an induced subgraph, 
  a contradiction (see \cite{computing}); while in the second case (when $p_1p_2$ 
  is new), the computing gives one possible graph shown in Fig.~\ref{fig-W4-case-2-2}.

\ssk

  %
  %
  \begin{figure}[ht]
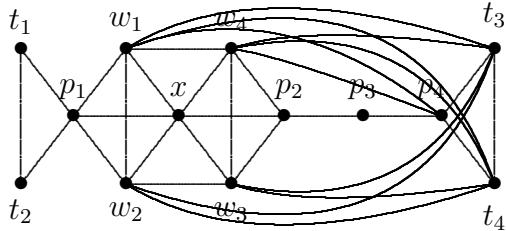

  $$\bp
  \setcoordinatesystem units <1.06mm,0.9mm>
  \setplotarea x from -10 to 10, y from -7 to 7
  \put{\beginpicture
  \setcoordinatesystem units <.7mm,.9mm>
  \setplotarea x from -12 to 12, y from -12 to 12
  \put{$\bullet$} at    -30   10
  \put{$t_1$} at    -30  14
  \put{$\bullet$} at    -30  -10
  \put{$t_2$} at    -30 -14
  \put{$\bullet$} at     0    0
  \put{$x$} at     0  4
  \put{$\bullet$} at   -20    0
  \put{$p_1$} at    -20   4
  \put{$\bullet$} at   -10   10
  \put{$w_1$} at    -10  14
  \put{$\bullet$} at   -10  -10
  \put{$w_2$} at    -10 -14
  \put{$\bullet$} at    10  -10
  \put{$w_3$} at     10 -14
  \put{$\bullet$} at    10   10
  \put{$w_4$} at     10  14
  \put{$\bullet$} at    20    0
  \put{$p_2$} at     21   4
  \put{$\bullet$} at    35    0
  \put{$p_3$} at     35   4
  \put{$\bullet$} at    50    0
  \put{$p_4$} at     48   4
  \put{$\bullet$} at    60   10
  \put{$t_3$} at     60  15
  \put{$\bullet$} at    60  -10
  \put{$t_4$} at     60 -15
  \plot  -20 0  -10 10  10 10  20 0  10 -10  -10 -10  -20 0 /
  \plot  -10 -10  -10 10  10 -10  10 10  -10 -10 /
  \plot  -20 0  20 0 /
  \plot 20 0   50 0  60 10  60 -10  50 0 /
  \plot -20 0  -30 -10  -30 10  -20 0 /
  \setquadratic
  \plotsymbolspacing=0.1pt
  \plot -10  10   20  12    50   0 /  
  \plot -10  10   35  12.5  60 -10 /  
  \plot -10  10   20  16    60  10 /  
  \plot  10  10   25   7    50   0 /  
  \plot  10  10   41   8    60 -10 /  
  \plot  10  10   30  12    60  10 /  
  \plot -10 -10   35 -12.5  60  10 /  
  \plot -10 -10   20 -16    60 -10 /  
  \plot  10 -10   41  -8    60  10 /  
  \plot  10 -10   30 -12    60 -10 /  
  \setlinear
  %
  %
  \endpicture} at   0  0
  \ep$$
  \bsm
  \caption{The only possibility in the second case of Subcase 2.2}
  \label{fig-W4-case-2-2}
  \end{figure}

  By the connectivity assumption, we have $d_G(p_3)\geq 3$, hence the vertex $p_3$ 
  has another neighbor $v\in V(G)\sm(V(F)\cup V(W))$. Since 
  $\lab p_3,p_2,p_4,v\rabg\niso\claw$, there must be some more edges.   
  Checking by computer all possibilities (see \cite{computing}),
  we again conclude that the resulting graph always contains at least one of 
  the graphs $\claw$, $\Gt$ or $W_5$ as an induced subgraph, a contradiction.
  \end{mylist}
  \vspace*{-12mm}
\end{mylist}~
\end{proof}

\bs

We now know that $\bG$ is $\{\claw,W_4,W_5\}$-free and we will repeatedly 
use this property.
For the sake of brevity, we introduce the following notion.
Given a graph $G$ and its vertex $x$, we say that $N_G(x)$ contains an 
{\em endgame} if $N_G(x)$ contains vertices $x_1,\ldots,x_k$ satisfying 
at least one of the following conditions (see also Fig.~\ref{fig-koncovky}):
\begin{mathitem}
\item $k=3$ and $\{x_1x_2,x_2x_3,x_3x_1\}\cap E(G)=\emptyset$,
\item $k=4$, $\{x_1x_2,x_2x_3,x_3x_4,x_4x_1\}\subset E(G)$ and 
      $\{x_1x_3,x_2x_4\}\cap E(G)=\emptyset$,
\item $k=5$ and $\{x_1x_2,x_2x_3,x_3x_4,x_4x_5,x_5x_1\}\cap E(G)=\emptyset$,

\item $k=5$, $\{x_1x_2,x_2x_3,x_3x_4,x_4x_5\}\cap E(G)=\emptyset$ and 
      $\{x_1x_4,x_2x_5\}\subset E(G)$,
\item $k=5$, $\{x_2x_3,x_3x_4,x_1x_5\}\cap E(G)=\emptyset$ and 
      $\{x_2x_5,x_5x_3,x_3x_1,x_1x_4\}\subset E(G)$.
\end{mathitem}

\newcommand{\x}{0.09mm}
\newcommand{\y}{0.09mm}

%
%
\begin{figure}[ht]
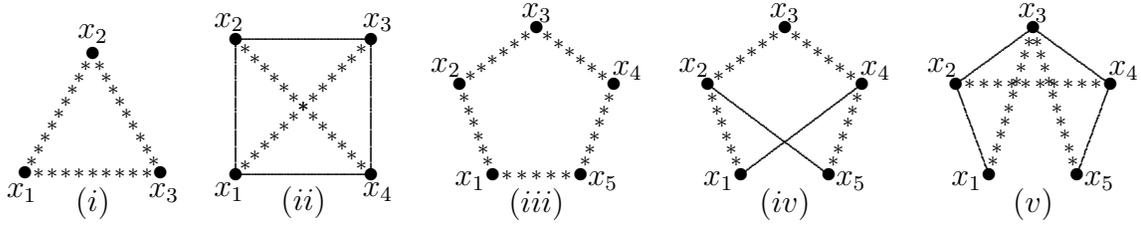

$$\bp
\setcoordinatesystem units <1mm,1mm>
\setplotarea x from -50 to 50, y from -7 to 7
\put{\beginpicture
\setcoordinatesystem units <\x,\y>
\setplotarea x from -120 to 120, y from -120 to 120
\put{$\bullet$} at    100  -97
\put{$\bullet$} at   -100  -97
\put{$\bullet$} at     0   80
\put{$x_1$} at    -105  -125
\put{$x_2$} at       0  110
\put{$x_3$} at     105  -125
\setplotsymbol ({\scriptsize $*$})
\plotsymbolspacing=5pt
\plot -100 -97  0 80  100 -97  -100 -97 /
\plotsymbolspacing=0.4pt
\setplotsymbol ({\fiverm .})
\put{$(i)$} at  0 -140
\endpicture} at -65  -1.3
\put{\beginpicture
\setcoordinatesystem units <\x,\y>
\setplotarea x from -120 to 120, y from -120 to 120
\put{$\bullet$} at   100   100
\put{$\bullet$} at   100  -100
\put{$\bullet$} at  -100  -100
\put{$\bullet$} at  -100   100
\put{$x_1$} at   -110 -125
\put{$x_2$} at   -110  125
\put{$x_3$} at    110  125
\put{$x_4$} at    110 -125
\plot  100 100   100 -100  -100 -100  -100 100  100 100 /
\setplotsymbol ({\scriptsize $*$})
\plotsymbolspacing=5pt
\plot  100 100   -100 -100  /
\plot  100 -100   -100 100  /
\plotsymbolspacing=0.4pt
\setplotsymbol ({\fiverm .})
\put{$(ii)$} at   0 -140
\endpicture} at -37  -0.5
\put{\beginpicture
\setcoordinatesystem units <\x,\y>
\setplotarea x from -120 to 120, y from -120 to 120
\put{$\bullet$} at     0 120
\put{$\bullet$} at   114  37
\put{$\bullet$} at    65 -97
\put{$\bullet$} at   -65 -97
\put{$\bullet$} at  -114 37
\put{$x_1$} at   -98 -105
\put{$x_2$} at  -134   62
\put{$x_3$} at     0  143
\put{$x_4$} at   134   57
\put{$x_5$} at    98 -105
\setplotsymbol ({\scriptsize $*$})
\plotsymbolspacing=5pt
\plot  0 120   114 37  65 -97  -65 -97  -114 37   0 120 /
\plotsymbolspacing=0.4pt
\setplotsymbol ({\fiverm .})
\put{$(iii)$} at  0 -140
\endpicture} at -6  0
\put{\beginpicture
\setcoordinatesystem units <\x,\y>
\setplotarea x from -120 to 120, y from -120 to 120
\put{$\bullet$} at     0 120
\put{$\bullet$} at   114  37
\put{$\bullet$} at    65 -97
\put{$\bullet$} at   -65 -97
\put{$\bullet$} at  -114 37
\put{$x_1$} at   -98 -105
\put{$x_2$} at  -134   62
\put{$x_3$} at     0  143
\put{$x_4$} at   134   57
\put{$x_5$} at    98 -105
\setplotsymbol ({\scriptsize $*$})
\plotsymbolspacing=5pt
\plot   -65 -97  -114 37  0 120  114 37  65 -97 /
\plotsymbolspacing=0.4pt
\setplotsymbol ({\fiverm .})
\plot -65 -97   114 37 /
\plot  65 -97  -114 37 /
\put{$(iv)$} at  0 -140
\endpicture} at  27  0
\put{\beginpicture
\setcoordinatesystem units <\x,\y>
\setplotarea x from -120 to 120, y from -120 to 120
\put{$\bullet$} at     0 120
\put{$\bullet$} at   114  37
\put{$\bullet$} at    65 -97
\put{$\bullet$} at   -65 -97
\put{$\bullet$} at  -114 37
\put{$x_1$} at   -98 -105
\put{$x_2$} at  -134   62
\put{$x_3$} at     0  143
\put{$x_4$} at   134   57
\put{$x_5$} at    98 -105
\setplotsymbol ({\scriptsize $*$})
\plotsymbolspacing=5pt
\plot  -114 37  114 37  /
\plot   -65 -97  0 120  65 -97 /
\plotsymbolspacing=0.4pt
\setplotsymbol ({\fiverm .})
\plot -65 -97  -114 37  0 120  114 37  65 -97 /
\put{$(v)$} at   0 -140
\endpicture} at  60  0
\ep$$
\bsm
\caption{The endgames used in Proposition~\ref{prop-koncovky}. The star-lines 
indicate pairs of nonadjacent vertices (i.e., edges in the complement of $G$).}
\label{fig-koncovky}
\end{figure}
%

%
%
\begin{proposition}
\label{prop-koncovky}
Let $G$ be a graph and let $x$ be its vertex such that $N_G(x)$ contains
an endgame. Then $G$ contains an induced $\claw$, $W_4$ or $W_5$ with center 
at $x$.
\end{proposition}

\begin{proof}
$(i)$ $\lab x,x_1,x_2,x_3\rabg\iso\claw$.

\ms

$(ii)$ $\lab x,x_1,x_2,x_3,x_4\rabg\iso W_4$.

\ms

$(iii)$  If, say, $x_1x_3\notin E(G)$, then the vertices 
$x_1,x_2,x_3$ satisfy $(i)$. 
Thus, by symmetry, $\{x_1x_3,x_3x_5,x_5x_2,x_2x_4,x_4x_1\}\subset E(G)$, 
and then $\lab x,x_1,x_3,x_5,x_2,x_4\rabg\iso W_5$.

\ms

$(iv)$ If $x_1x_5\notin E(G)$, then $x_1,\ldots,x_5$ satisfy $(iii)$, 
hence $x_1x_5\in E(G)$. Now, if $x_2x_4\notin E(G)$, then $x_2,x_3,x_4$
satisfy $(i)$, and if $x_2x_4\in E(G)$, then $x_1,x_4,x_2,x_5$ 
satisfy $(ii)$.

\ms

$(v)$ If $x_2x_5\notin E(G)$, then $x_1,x_3,x_5,x_2,x_4$ satisfy $(iv)$, 
and if $x_2x_5\in E(G)$, then $x_2,x_3,x_4,x_5$ satisfy $(ii)$.
\end{proof}

%
%
\begin{proposition}
\label{prop-closure-P6}
Let $G$ be a 3-connected $\{\claw,\Gt\}$-free graph and let $\bG$ be its 
$\Gt$-closure. Then $\bG$ is $\{P_6^2,P_6^{2+}\}$-free.
\end{proposition}

\begin{proof}
Let, to the contrary, $Q=\lab v_0,\ldots,v_5\rabgb\indsub\bG$ be such that 
$Q\iso P_6^2$ or $Q\iso P_6^{2+}$ (with the labeling of vertices as in 
Fig.~\ref{fig-kola_a_cesty}).
We will prove the statement for both $P_6^2$ and $P_6^{2+}$ at the same time 
since our proof does not depend on whether the edge $v_0v_5$ is in $Q$
or not.
We show that either none of the vertices $v_1,v_4$ is a vertex of degree~4
in an induced subgraph $F\iso S$, or for at least one $v_i\in\{v_1,v_4\}$, 
$\lab\NbG(v_i)\ragb$ is 2-connected. By Corollary~\ref{coro-subgraph_S}
and by Theorem~\ref{thmA-HC-2-conn_neighb}, this will imply that at least 
one $v_i\in\{v_1,v_4\}$ is feasible, and hence $\Gst_{v_i}$ contains an 
induced $\Gt$. Showing that this is not possible, we obtain the requested 
contradiction.

For the vertices of the induced subgraph $F\iso S$, we will use the labeling 
as in Fig.~\ref{fig-graph_S}, and we will use $C$ to denote its central 
triangle $C=z_1z_2z_3$. By symmetry, we suppose that $v_1\in V(C)$, and we 
distinguish several cases according to the mutual position of $C$ and $Q$.

\begin{mylist}

  \item[\underline{\bf Case 1:}] {\sl $\{v_1,v_3\}\subset V(C)$.} \quad \\
  Up to a symmetry, set $v_1=z_1$ and $v_3=z_2$. Then there is a vertex 
  $w_1\in V(F)\sm V(Q)$ such that $w_1v_1,w_1v_3\in E(\bG)$: 
  if $z_3\notin V(F)$, we simply set $w_1=z_3$; otherwise necessarily 
  $z_3=v_2$, and we set $w_1=z_5$. 
  
  \begin{mylist}
    \item[\underline{\bf Subcase 1.1:}] {\sl $w_1v_5\in E(\bG)$.} \quad \\
    If $w_1v_4\notin E(\bG)$, then $w_1v_4v_1v_5v_2$ is endgame $(iv)$ in 
    $\NbG(v_3)$ (note that $v_4v_1,v_1v_5,v_5v_2\notin E(\bG)$ since 
    $Q\indsub\bG$; similarly in the following), hence $w_1v_4\in E(\bG)$. 
    If $w_1v_2\notin E(\bG)$, then $w_1v_4v_2v_1$ is endgame $(ii)$ in 
    $\NbG(v_3)$, hence $w_1v_2\in E(\bG)$.
    This implies that $v_2\notin V(F)$ (since $F\iso S$ is induced).
    Thus, there is a vertex $w_2\in V(F)\sm V(Q)$ such that 
    $w_1v_1,w_2v_3\in E(\bG)$.
    If $w_2v_5\in E(\bG)$, then $w_2v_5w_1v_1$ is endgame $(ii)$ in 
    $\NbG(v_3)$, hence $w_2v_5\notin E(\bG)$.
    Since $\lab v_3,v_2,w_2,v_5\rabgb\niso\claw$, $v_2w_2\in E(\bG)$. 
    Now we have $v_0w_1\notin E(\bG)$ and $v_0w_2\notin E(\bG)$, for otherwise 
    $\la\NbG(v_1)\ragb$ is 2-connected by Lemma~\ref{lemmaA-2_indep_sets}
    and $v_1$ is feasible.
    But then $\lab v_1,w_1,v_0,w_2\rabgb\iso\claw$, a contradiction.
   
    \item[\underline{\bf Subcase 1.2:}] {\sl $w_1v_5\notin E(\bG)$.} \quad \\
    Since $\lab v_3,w_1,v_5,v_2\rabgb\niso\claw$, we have $w_1v_2\in E(\bG)$.
    This implies that $v_2\notin V(F)$, hence there is again a vertex 
    $w_2\in V(F)\sm V(Q)$ such that $w_2v_1,w_2v_3\in E(\bG)$. 
    Note that $w_1w_2\notin E(\bG)$ since $F$ is induced. Then 
    $w_2v_5\in E(\bG)$ since otherwise $\lab v_3,w_1,w_2,v_5\rabgb\iso\claw$.
    If $v_0w_2\in E(\bG)$, then $\la\NbG(v_1)\ragb$ is 2-connected by 
    Lemma~\ref{lemmaA-2_indep_sets}, hence $v_0w_2\notin E(\bG)$, and from 
    $\lab v_1,v_0,w_1,w_2\rabgb\niso\claw$ we have $v_0w_1\in E(\bG)$.
    Now, $v_2w_2\notin E(\bG)$ for otherwise $\la\NbG(v_1)\ragb$ is 2-connected 
    by Lemma~\ref{lemmaA-2_indep_sets}.
    If $w_2v_4\notin E(\bG)$, then $w_2v_4v_1v_5w_1$ is endgame $(iii)$ in 
    $\NbG(v_3)$, hence $w_2v_4\in E(\bG)$. However, then $w_2v_4v_2v_1$ is 
    endgame $(ii)$ in $\NbG(v_3)$, a contradiction.
  \end{mylist}

  \item[\underline{\bf Case 2:}] {\sl $\{v_1,v_2\}\subset V(C)$.} \quad 
  
  \begin{mylist}
    \item[\underline{\bf Subcase 2.1:}] {\sl $v_0\in V(C)$.} \quad \\
    By the structure of $F$, there is a vertex $w_1\in V(F)\sm V(Q)$ such that 
    $w_1v_0,w_1v_2\in E(F)$. Since $w_1v_1\notin E(\bG)$ ($F$ is induced), 
    $v_1v_4\notin E(\bG)$ ($Q$ is induced), and $\lab v_2,w_1,v_1,v_4\rabgb\niso\claw$, 
    we have $w_1v_4\in E(\bG)$. Then $w_1v_1v_4v_0v_3$ is endgame $(iv)$ in $\NbG(v_2)$.
    
    \item[\underline{\bf Subcase 2.2:}] {\sl $v_0\notin V(C)$.} \quad \\
    Note that also $v_3\notin V(C)$ since otherwise we are back in Case~1.
    Then there is a vertex $w_1\in V(C)\sm V(Q)$ such that $w_1v_1,w_1v_2\in E(C)$. 
    
    Suppose that $w_1v_3\notin E(\bG)$. Then necessarily $w_1v_4\in E(\bG)$, for
    otherwise the subgraph $Q'=\lab w_1,v_1,v_2,v_3,v_4,v_5\rabgb$ is isomorphic 
    to $P_6^2$ or to $P_6^{2+}$ and, for $Q'$, we are back in Case~1; however, then 
    $w_1v_1v_3v_4$ is endgame $(ii)$ in $\NbG(v_2)$, a contradiction.
    
    Thus, $w_1v_3\in E(\bG)$, implying that $v_3\notin V(F)$. Hence there is a vertex 
    $w_2\in V(F)\sm V(Q)$ such that $w_2w_1,w_2v_1\in E(F)$. 
    If $w_2v_0\in E(\bG)$, then $\la\NbG(v_1)\ragb$ is 2-connected by 
    Lemma~\ref{lemmaA-2_indep_sets}, hence $w_2v_0\notin E(\bG)$.
    Since $\lab v_1,v_0,w_2,v_3\rabgb\niso\claw$, we have $w_2v_3\in E(\bG)$.
    Then $w_2v_2v_5v_1v_4$ is endgame $(iv)$ in $\NbG(v_3)$, a contradiction.
  \end{mylist}

  \item[\underline{\bf Case 3:}] {\sl $\{v_0,v_1\}\subset V(C)$, $v_2\notin V(C)$.} \\
  Then there is a vertex $w_1\in V(C)\sm V(Q)$ such that $w_1v_0,w_1v_1\in E(C)$.
  
  \begin{mylist}
    \item[\underline{\bf Subcase 3.1:}] {\sl $w_1v_2\notin E(\bG)$.} \quad \\
    If $w_1v_3\in E(\bG)$, then $w_1v_0v_2v_3$ is endgame $(ii)$ in $\NbG(v_1)$, 
    hence $w_1v_3\notin E(\bG)$.
    Then $v_3$ cannot have two neighbors in $V(C)$, implying $v_3\notin V(F)$. 
    Hence there is a vertex $w_2\in V(F)\sm V(Q)$ such that $w_2v_1,w_2w_1\in E(F)$. 
    Now, $w_2v_3\notin E(\bG)$ for otherwise $\la\NbG(v_1)\ragb$ is 2-connected by 
    Lemma~\ref{lemmaA-2_indep_sets}, and then $\lab v_1,w_2,v_0,v_3\rabgb\iso\claw$, 
    a contradiction.

    \item[\underline{\bf Subcase 3.2:}] {\sl $w_1v_2\in E(\bG)$.} \quad \\
    Then $|\NbG(v_2)\cap V(C)|=3$, implying $v_2\notin V(F)$. Hence there is a vertex 
    $w_2\in V(F)\sm V(Q)$ with $w_2v_0,w_2v_1\in E(F)$.
    If $w_2v_3\in E(\bG)$, then $\la\NbG(v_1)\ragb$ is 2-connected by 
    Lemma~\ref{lemmaA-2_indep_sets}, hence $w_2v_3\notin E(\bG)$, and from 
    $\lab v_1,w_1,w_2,v_3\rabgb\niso\claw$ we have $w_1v_3\in E(\bG)$. 
    Now, $v_5\in V(F)$ would imply $v_5w_1,v_5v_0\in E(\bG)$, and then $v_0v_2v_3v_5$
    is endgame $(ii)$ in $\NbG(w_1)$; hence $v_5\notin V(F)$.
    Then there is a vertex $w_3\in V(F)\sm V(Q)$ with $w_3w_1,w_3v_0\in E(F)$
    (clearly $v_2,v_3,v_4\notin V(F)$).
    Now $w_2v_2\notin E(\bG)$ for otherwise $\la\NbG(v_1)\ragb$ is 2-connected by 
    Lemma~\ref{lemmaA-2_indep_sets}, from $\lab v_0,w_3,v_2,w_2\rabgb\niso\claw$ 
    we have $w_3v_2\in E(\bG)$, and from $\lab v_2,w_3,v_1,v_4\rabgb\niso\claw$ 
    we have $w_3v_4\in E(\bG)$.
    
    If $w_3v_3\notin E(\bG)$, then $v_0v_1v_3v_4w_3$ is endgame $(v)$ in $\NbG(v_2)$,
    hence $w_3v_3\in E(\bG)$. This implies that $v_3\notin V(F)$, hence there is a 
    vertex $w_4\in V(F)\sm V(Q)$ such that $w_1v_1,w_4w_1\in E(F)$. Since
    $\lab v_1,v_0,w_4,v_3\rabgb\niso\claw$, $w_4v_3\in E(\bG)$. Since 
    $\lab v_1,w_4,v_2,w_2\rabgb\niso\claw$, $w_4v_2\in E(\bG)$, and since 
    $\lab v_2,v_0,w_4,v_4\rabgb\niso\claw$, $w_4v_4\in E(\bG)$.
    Now, if $w_4v_5\in E(\bG)$ or $w_3v_5\in E(\bG)$, then $\la\NbG(v_4)\ragb$ is 
    2-connected by Lemma~\ref{lemmaA-2_indep_sets} (the independent sets are
    $\{v_2,v_5\}$ and $\{w_3,v_4\}$), hence $w_4v_5\notin E(\bG)$ and 
    $w_3v_5\notin E(\bG)$. However, then we have 
    $\lab v_3,w_4,w_3,v_5\rabgb\iso\claw$, 
    a contradiction.
  \end{mylist}

  \item[\underline{\bf Case 4:}] {\sl $\{v_0,v_3\}\cap V(C)=\emptyset$.} \\
  Then $C=v_1w_1w_2$ for some $w_1,w_2\in V(F)\sm V(Q)$. 
  If $|V(F)\cap\{v_0,v_2,v_3\}|\geq 2$, then $v_0\in V(F)$ and $v_3\in V(F)$
  (since vertices of degree 2 in $F$ are independent), and $\la\NbG(v_1)\ragb$ is
  2-connected by Lemma~\ref{lemmaA-2_indep_sets}. Thus, at most one of the vertices 
  $v_0,v_2,v_3$ is in $V(F)$. This implies that there is a vertex 
  $w_3\in V(F)\sm(V(Q)\cup\{w_1,w_2\})$ such that $w_3v_1,w_3w_2\in E(F)$. Since 
  $\lab v_1,w_3,w_1,v_0\rabgb\niso\claw$, up to a symmetry, $w_3v_0\in E(\bG)$. 
  Then $w_1v_3\notin E(\bG)$, for otherwise $\la\NbG(v_1)\ragb$ is 2-connected
  by Lemma~\ref{lemmaA-2_indep_sets}.
  Since $\lab v_1,w_3,w_1,v_3\rabgb\niso\claw$, we have $w_3v_3\in E(\bG)$, and
  since $\lab v_1,v_0,w_1,v_3\rabgb\niso\claw$, we have $v_0w_1\in E(\bG)$.
  But then $\la\NbG(v_1)\ragb$ is 2-connected by Lemma~\ref{lemmaA-2_indep_sets}.
\end{mylist}

Thus, we know that at least one of the vertices $v_1,v_4$ is feasible. We choose 
the notation such that $v_1$ is feasible. By the definition of the $\Gt$-closure,
$\bGst_{v_1}$ contains an induced $\Gt$. To reach a contradiction, we show that 
this is not possible.

\bs

Set $x=v_1$, let $F=\lab t_1,t_2,p_1,p_2,p_3,p_4,t_3,t_4\rab_{\bGstx}\iso\Gt$, 
and let $y_1y_2\in E(F)$ be a new edge. Since $t_1t_2$ cannot be the only new edge
(then would be $\lab p_1,t_1,t_2,p_2\rabgb\iso\claw$), by symmetry, we can 
assume that $y_1y_2\in\{t_2p_1,p_1p_2,p_2p_3\}$.

\begin{mylist}

  \item[\underline{\bf Case 1:}] {\sl $\dist_{\bG}(y_1,y_2)=2$.} \quad \\
  Then, by Lemma~\ref{lemma-W5-dist3}, $y_1y_2\in\{p_1p_2,p_2p_3\}$.
  Let $y_1z_1y_2$ be a shortest $(y_1,y_2)$-path in $\la\NbG(x)\ragb$.

  %
  %
  \setcounter{prostrclaim}{0}
  \begin{claim}
  \label{claim-P6-1} 
  $\NbG(z_1)\cap V(F)=\{y_1,y_2\}$.
  \end{claim}
  
  \begin{proofcl}
  Let first $y_1y_2=p_1p_2$, and let $z_1z_2\in E(\bG)$ for some 
  $z_2\in V(F)\sm\{p_1,p_2\}$. If $z_2\in\{p_3,t_3,t_4\}$, then 
  $\lab z_1,p_1,p_2,z_2\rabgb\iso\claw$, and if both $z_1p_3\in E(\bG)$ and 
  $z_2t_i\in E(\bG)$ for some $i\in\{1,2\}$, then $\lab z_1,t_i,x,p_3\rabgb\iso\claw$.
  Hence, by symmetry, either $z_1p_3\in E(\bG)$, or, say, $z_1t_2\in E(\bG)$, 
  but in the first case $\lab t_1,t_2,p_1,z_1,p_3,p_4,t_3,t_4\rabgb\iso\Gt$, 
  and in the second case $\lab t_2,p_1,z_1,p_2,p_3,p_4,t_3,t_4\rabgb\iso\Gt$.
  
  Let next $y_1y_2=p_2p_3$, and let $z_1z_2\in E(\bG)$ for some 
  $z_2\in V(F)\sm\{p_2,p_3\}$. If $z_2\in\{t_1,t_2,t_3,t_4\}$, then 
  $\lab z_1,z_2,p_2,p_3\rabgb\iso\claw$, and if both $z_1p_1\in E(\bG)$ and 
  $z_1p_4\in E(\bG)$, then $\lab z_1,p_1,x,p_4\rabgb\iso\claw$. Thus, by symmetry,
  we can assume that $z_1p_1\in E(\bG)$ and   $z_1p_4\notin E(\bG)$, but then 
  $\lab t_1,t_2,p_1,z_1,p_3,p_4,t_3,t_4\rabgb\iso\Gt$, a contradiction.
  \end{proofcl}

\bsm

  %
  %
  \begin{claim}
  \label{claim-P6-2}
  $\NbG[x]=\NbG[z_1]$.
  \end{claim}

\bsm

  \begin{proofcl}
  By Claim~\ref{claim-P6-1}, $\NbG(x)\cap V(F)=\NbG(z_1)\cap V(F)$.

\begin{mylist}
  \item[\underline{Case Cl-\ref{claim-P6-2}-1:}] {\sl $y_1y_2=p_1p_2$.} \\
  We first show that $\NbG[x]\subset\NbG[z_1]$. Let thus, to the contrary, 
  $z_2\in\NbG[x]\sm\NbG[z_1]$.
  Since $\lab x,p_1,p_2,z_2\rabgb\niso\claw$, 
  $p_1z_2\in E(\bG)$ or $p_2z_2\in E(\bG)$. On the other hand, if both 
  $p_1z_2\in E(\bG)$ and $p_2z_2\in E(\bG)$, then we have a contradiction with 
  Claim~\ref{claim-P6-1}. Hence $z_2$ is adjacent to exactly one of $p_1,p_2$.
  Let first $p_1z_2\in E(\bG)$ and $p_2z_2\notin E(\bG)$. 
  By Claim~\ref{claim-P6-1}, $\lab p_1,t_1,z_1,z_2\rabgb\niso\claw$ implies
  $t_1z_2\in E(\bG)$. From $\lab p_1,z_2,x,p_2,p_3,p_4,t_3,t_4\rabgb\niso\Gt$
  we have $z_2z_3\in E(\bG)$ for some $z_3\in\{p_3,p_4,t_3,t_4\}$, but then
  $\lab z_2,t_1,x,z_3\rabgb\iso\claw$, a contradiction.
  Thus, we have $p_2z_2\in E(\bG)$ and $p_1z_2\notin E(\bG)$.
  By Claim~\ref{claim-P6-1}, $\lab p_2,z_1,z_2,p_3\rabgb\niso\claw$ implies
  $p_3z_2\in E(\bG)$. From $\lab p_1,z_1,x,z_2,p_3,p_4,t_3,t_4\rabgb\niso\Gt$
  we have $z_2z_3\in E(\bG)$ for some $z_3\in\{p_4,t_3,t_4\}$. However, if 
  $z_3\in\{t_3,t_4\}$, then $\lab z_2,z_3,p_3,x\rabgb\iso\claw$, hence 
  $z_2p_4\in E(\bG)$, but then $\lab t_1,t_2,p_1,x,z_2,p_4,t_3,t_4\rabgb\iso\Gt$, 
  a contradiction.
  Thus, $\NbG[x]\subset\NbG[z_1]$.

  Next we show that also $\NbG[z_1]\subset\NbG[x]$. Let, to the contrary, 
  $z_2\in\NbG[z_1]\sm\NbG[x]$.
  Let first $p_1z_2\in E(\bG)$. From $\lab p_1,t_1,x,z_2\rabgb\niso\claw$ then 
  $t_1z_2\in E(\bG)$. 
  Since $\lab p_1,z_2,z_1,p_2,p_3\car p_4,t_3,t_4\rabgb\niso\Gt$, we have 
  $z_2z_3\in E(\bG)$ for some $z_3\in\{p_2,p_3,p_4,t_3,t_4\}$; however, 
  $z_3\in\{p_3,p_4,t_3,t_4\}$ implies $\lab z_2,z_3,t_1,z_1\rabgb\iso\claw$, 
  and $z_3=p_2$ implies $\lab p_2,p_3,z_2,x\rabgb\iso\claw$.
  Thus, $p_1z_2\notin E(\bG)$, and from $\lab z_1,p_1,p_2,z_2\rabgb\niso\claw$
  we have $z_2p_2\in E(\bG)$. Since $\lab p_2,x,z_2,p_3\rabgb\niso\claw$, 
  also $z_2p_3\in E(\bG)$. Now, since 
  $\lab p_1,x,z_1,z_2,p_3,p_4,t_3,t_4\rabgb\niso\Gt$, by Claim~\ref{claim-P6-1},
  we have $z_2z_3\in E(\bG)$ for some $z_3\in\{p_4,t_3,t_4\}$. However, 
  $z_3\in\{t_3,t_4\}$ implies $\lab z_2,z_1,p_3,z_3\rabgb\iso\claw$, 
  and for $z_3=p_4$ we have $\lab t_1,t_2,p_1,z_1,z_2,p_4,t_3,t_4\rabgb\iso\Gt$, 
  a contradiction.

  \item[\underline{Case Cl-\ref{claim-P6-2}-2:}] {\sl $y_1y_2=p_2p_3$.} \\
  We again first show that $\NbG[x]\subset\NbG[z_1]$. Let, to the contrary, 
  $z_2\in\NbG[x]\sm\NbG[z_1]$.
  Since $\lab x,p_1,p_2,z_2\rabgb\niso\claw$, up to a symmetry, $p_2z_2\in E(\bG)$, 
  and since $\lab p_2,p_1,z_1,z_2\rabgb\niso\claw$, by Claim~\ref{claim-P6-1},
  $p_1z_2\in E(\bG)$. Since $\lab p_1,z_2,p_2,z_1,p_3,p_4,t_3,t_4\rabgb\niso\Gt$ 
  and by Claim~\ref{claim-P6-1}, $z_2z_3\in E(\bG)$ for some 
  $z_3\in\{p_3,p_4,t_3,t_4\}$. However, for $z_3\in\{p_4,t_3,t_4\}$,
  $\lab z_2,p_1,x,z_3\rabgb\iso\claw$, hence $z_3=p_3$, but then 
  $\lab t_1,t_2,p_1,z_2,p_3,p_4,t_3,t_4\rabgb\iso\Gt$, a contradiction.
  Thus, $\NbG[x]\subset\NbG[z_1]$.
  
  Finally, we show that also $\NbG[z_1]\subset\NbG[x]$. Let thus again, to the 
  contrary, $z_2\in\NbG[z_1]\sm\NbG[x]$.
  Since $\lab z_1,p_2,p_3,z_2\rabgb\niso\claw$, up to a symmetry, 
  $z_2p_2\in E(\bG)$, and since $\lab p_2,p_1,z_2,x\rabgb\niso\claw$, also 
  $z_2p_1\in E(\bG)$. Since $\lab p_1,z_2,p_2,x,p_3,p_4,t_3,t_4\rabgb\niso\Gt$, 
  we have $z_2z_3\in E(\bG)$ for some $z_3\in\{p_3,p_4,t_3,t_4\}$; however, 
  if $z_3\in\{p_4,t_3,t_4\}$, then $\lab z_2,p_1,z_1,z_3\rabgb\iso\claw$ by
  Claim~\ref{claim-P6-1}, and if $z_3=p_3$, then 
  $\lab p_3,z_2,x,p_4\rabgb\iso\claw$, a contradiction.
  \end{mylist}
  \vspace*{-9mm}
  ~
  \end{proofcl}
 
 \bsm

  By the assumption, set $Q=\lab v_0,v_1,v_2,v_3,v_4,v_5\rabgb$, where 
  $Q\iso P_6^2$ or $Q\iso P_6^{2+}$ and $x=v_1$. Since $\bG$ is 3-connected 
  and not Hamilton-connected, by Theorem~\ref{thmA-Chvatal-Erdos}, 
  $\alpha(\bG)\geq 3$. Thus, by Theorem~\ref{thmA-Fouquet} and by
  Proposition~\ref{prop-closure-W5}, $\NbG(x)$, hence also $\NbG[x]$,
  can be covered by two cliques, say, $K_1$ and $K_2$. 
  Since $y_1,y_2\in\NbG(x)$ and $y_1y_2\notin E(\bG)$ (where $y_1=p_1$ and 
  $y_2=p_2$, or $y_1=p_2$ and $y_2=p_3$, depending on the case), we can choose 
  the notation such that $y_1\in K_1\sm K_2$ and $y_2\in K_2\sm K_1$. 
  Clearly $v_1=x\in K_1\cap K_2$, and, by Claim~\ref{claim-P6-2}, also 
  $z_1\in K_1\cap K_2$.
  On the other hand, since $v_4\in \NbG(v_2)\sm \NbG(v_1)$, by 
  Claim~\ref{claim-P6-2}, $v_2\notin K_1\cap K_2$. 
  Hence either $v_2\in K_1\sm K_2$ and, since $v_0,v_3\in\NbG(v_1)$ but 
  $v_0v_3\notin E(\bG)$, one of $v_0,v_3$ is in $K_2\sm K_1$, or,
  symmetrically, $v_2\in K_2\sm K_1$ and one of $v_0,v_3$ is in $K_1\sm K_2$.
  Thus, we conclude that there are vertices $w_1\in K_1\sm K_2$ and 
  $w_2\in K_2\sm K_1$ such that $w_1,w_2\in V(Q)$ and $w_1w_2\in E(Q)$. 
  By Claim~\ref{claim-P6-2}, $w_1y_2\notin E(\bG)$ and $w_2y_1\notin E(\bG)$.
  
  \begin{mylist}
    \item[\underline{\bf Subcase 1.1:}] {\sl $y_1y_2=p_1p_2$.} \quad \\
    Since $\lab p_1,w_1,x,p_2,p_3,p_4,t_3,t_4\rabgb\niso\Gt$, 
    $w_1z_2\in E(\bG)$ for some $z_2\in\{p_3,p_4,t_3,t_4\}$. Since
    $\lab w_1,p_1,w_2,z_2\rabgb\niso\claw$, also $w_2z_2\in E(\bG)$. 
    If $z_3=p_3$, then $w_1p_3p_2x$ is endgame $(ii)$ in $\NbG(w_2)$, hence 
    $z_2\in\{p_4,t_3,t_4\}$.
    Recall that $w_1,w_2\in V(Q)$ and either $w_1=v_2$ and $w_2\in\{v_0,v_3\}$, 
    or $w_1\in\{v_0,v_3\}$ and $w_2=v_2$. Then, if $w_1=v_2$ and $w_2=v_0$, 
    $z_2v_1v_4v_0v_3$ is endgame $(iv)$ in $\NbG(v_2)$, and if $w_1=v_2$ and 
    $w_2=v_3$, then $p_2v_2v_5v_1v_4$ is endgame $(iv)$ in $\NbG(v_3)$. 
    The remaining two cases $w_2=v_2$ are symmetric (since our argument did 
    not use the vertices $t_1,t_2,p_3,p_4,t_3,t_4$).
    
    \item[\underline{\bf Subcase 1.2:}] {\sl $y_1y_2=p_2p_3$.} \quad \\
    Since $w_1,w_2\in V(Q)$, by symmetry, we can assume that $w_1=v_2$ and 
    $w_2\in\{v_0,v_3\}$. However, if $w_2=v_3$, then $p_3v_2v_5v_1v_4$ is 
    endgame $(iv)$ in $\NbG(v_3)$. Hence $w_2=v_0$.
    
    Next observe that a vertex $t\in\{t_1,t_2,t_3,t_4\}$ is adjacent to either 
    both the vertices $v_0,v_2$, or to none of them: if, say, $t_1v_2\in E(\bG)$, 
    then $\lab v_2,t_1,p_2,v_0\rabgb\niso\claw$ implies $t_1v_0\in E(\bG)$, and,
    conversely, if $t_1v_0\in E(\bG)$, then $\lab v_0,t_1,p_3,v_2\rabgb\niso\claw$ 
    implies $t_1v_2\in E(\bG)$. However, if $tv_0,tv_2\in E(\bG)$ for some 
    $t\in\{t_1,t_2,t_3,t_4\}$, then $tv_1v_4v_0v_3$ is endgame $(iv)$ in
    $\NbG(v_2)$. Thus, there are no edges between $\{v_0,v_2\}$ and 
    $\{t_1,t_2,t_3,t_4\}$.
    
    Now, if $v_0p_1\in E(\bG)$, then $\lab p_1,t_1,p_2,v_0\rabgb\iso\claw$, 
    hence $v_0p_1\notin E(\bG)$, and, symmetrically, $v_2p_4\notin E(\bG)$.
    Thus, the only possible edges between $\{v_0,v_2\}$ and 
    $V(F)\sm\{p_2,p_3\}$ are the edges $v_2p_1$ and $v_0p_4$. If both are 
    present, then $\lab t_1,t_2,p_1,v_2,v_0,p_4,t_3,t_4\rabgb\iso\Gt$, 
    and if, say, $v_2p_1\in E(\bG)$ but $v_0p_4\notin E(\bG)$, then 
    $\lab p_1,p_2,v_2,v_0,p_3,p_4,t_3,t_4\rabgb\iso\Gt$. Thus, there are no edges
    between $\{v_0,v_2\}$ and $V(F)\sm\{p_2,p_3\}$.
    
    If $v_3=p_2$, then, by the above observations, $v_4\neq p_1$, and from 
    $\lab v_3,p_1,v_4,v_1\rabgb\niso\claw$ we have $p_1v_4\in E(\bG)$. Since 
    $\lab p_1,v_4,v_3,v_1,p_3,p_4,t_3,t_4\rabgb\niso\Gt$, we have 
    $v_4z\in E(\bG)$ for some $z\in\{p_3,p_4,t_3,t_4\}$, but in each of these 
    cases, $\lab v_4,p_1,v_2,z\rabgb\iso\claw$, a contradiction.
    Thus, $v_3\neq p_2$.
    
    Since $\lab v_2,v_4,p_2,v_0\rabgb\niso\claw$, we have $p_2v_4\in E(\bG)$, 
    and from $\lab p_2,p_1,v_4,v_1\rabgb\niso\claw$ also $p_1v_4\in E(\bG)$.
    Since $\lab p_1,v_4,p_2,v_1,p_3,p_4,t_3,t_4\rabgb\niso\Gt$, we have 
    $v_4z\in E(\bG)$ for some $z\in\{p_3,p_4,t_3,t_4\}$, but in each of these 
    cases, $\lab v_4,p_1,v_2,z\rabgb\iso\claw$, a contradiction.
  \end{mylist}

  \item[\underline{\bf Case 2:}] {\sl $\dist_{\bG}(y_1,y_2)=3$.} \quad \\
  Let $K_1,K_2$ be the two cliques that cover $\NbG[x]$, chosen such that 
  $y_1\in K_1$ and $y_2\in K_2$. By the assumption of the case, 
  $K_1\cap K_2=\{x\}$. Since $x=v_1$, by the structure of $Q$, there are vertices 
  $w_1\in K_1\sm K_2$ and $w_2\in K_2\sm K_1$ such that $w_1,w_2\in V(Q)$, 
  $w_1w_2\in E(Q)$, and either $w_1=v_2$ and $w_2\in\{v_0,v_3\}$, or 
  $w_1\in\{v_0,v_3\}$ and $w_2=v_2$.
  Note that also $y_iw_i\in E(\bG)$, $i=1,2$, and $y_1w_2,y_2w_1\notin E(\bG)$.
  
  Now, if, say, $w_1=v_2$ and $w_2=v_3$, then $y_2v_2v_5v_1v_4$ is endgame $(iv)$
  in $\NbG(v_3)$. Since the case $w_1=v_3$ and $w_2=v_2$ is symmetric, 
  we conclude that either $w_1=v_2$ and $w_2=v_0$, or $w_1=v_0$ and $w_2=v_2$.
  
  \begin{mylist}
    \item[\underline{\bf Subcase 2.1:}] {\sl $y_1y_2=t_2p_1$.} \quad \\
    We show that $w_1,w_2$ have a common neighbor $z\in\{p_3,p_4,t_3,t_4\}$. 
    
    Since $\lab v_1,w_2,p_1,p_2,p_3,p_4,t_3,t_4\rabgb\niso\Gt$, $w_2z\in E(\bG)$
    for some $z\in\{p_2,p_3,p_4,t_3,t_4\}$. If $z\in\{p_3,p_4,t_3,t_4\}$, then
    $\lab w_2,w_1,p_1,z\rabgb\niso\claw$ implies that also $w_1z\in E(\bG)$; thus, 
    let $z=p_2$. Then from $\lab w_1,v_1,w_2,p_2,p_3,p_4,t_3,t_4\rabgb\niso\Gt$ 
    we have $w_1z'\in E(\bG)$ for some $z'\in\{p_3,p_4,t_3,t_4\}$, but then 
    $\lab w_1,t_2,w_2,z'\rabgb\niso\claw$ implies that also $w_2z'\in E(\bG)$. 
    Thus, in all cases, there is $z\in\{p_3,p_4,t_3,t_4\}$ such that 
    $w_1z,w_2z\in E(\bG)$. Since $\{w_1,w_2\}=\{v_0,v_2\}$, necessarily 
    $z\notin\{v_3,v_4\}$, and then $zv_1v_4v_0v_3$ is endgame $(iv)$ in 
    $\NbG(v_2)$.
    
    \item[\underline{\bf Subcase 2.2:}] {\sl $y_1y_2=p_1p_2$.} \quad \\
    We similarly show that $w_1$ and $w_2$ have a common neighbor 
    $z\in\{p_4,t_3,t_4\}$. 
    Since $\lab p_1,w_1,v_1,p_2,p_3,p_4,t_3,t_4\rabgb\niso\Gt$, 
    we have $w_1z\in E(\bG)$ for some $z\in\{p_3,p_4,t_3,t_4\}$.
    
    If $z\in\{p_4,t_3,t_4\}$, then $\lab w_1,p_1,w_2,z\rabgb\niso\claw$ implies
    that also $w_2z\in E(\bG)$; thus, let $z=p_3$. Then 
    $\lab p_3,p_2,p_4,w_1\rabgb\niso\claw$ implies that also $w_1p_4\in E(\bG)$, 
    and we are in the previous case. Thus, we have $w_1z,w_2z\in E(\bG)$ for 
    some $z\in\{p_4,t_3,t_4\}$. Again, $z\notin\{v_3,v_4\}$ since 
    $\{w_1,w_2\}=\{v_0,v_2\}$, and then $zv_1v_4v_0v_3$ is endgame $(iv)$
    in $\NbG(v_2)$.

    \item[\underline{\bf Subcase 2.3:}] {\sl $y_1y_2=p_2p_3$.} \quad \\
    We proceed in a similar way as in Subcase~1.2.
    If, say, $w_1t_i\in E(\bG)$ for some $i\in\{1,2,3,4\}$, then 
    $\lab w_1,w_2,p_2,t_i\rabgb\niso\claw$ implies that also $w_2t_i\in E(\bG)$
    and $t_iv_1v_4v_0v_3$ is endgame $(iv)$ in $\NbG(v_2)$; and if, say, 
    $w_2p_4\in E(\bG)$, then since $\lab p_4,t_4,p_3,w_1\rabgb\niso\claw$, 
    $w_1t_4\in E(\bG)$, and we are in the previous case. Hence $w_1p_1$ and 
    $w_2p_4$ are the only possible edges between $\{w_1,w_2\}$ and 
    $V(F)\sm\{p_2,p_3\}$; however, if both are present, then 
    $\lab t_1,t_2,p_1,w_1,w_2,p_4,t_3,t_4\rabgb\iso\Gt$, and if, say, 
    $w_1p_1\in E(\bG)$ but $w_2p_4\notin E(\bG)$, then 
    $\lab p_1,p_2,w_1,w_2,p_3,p_4,t_3,t_4\rabgb\iso\Gt$.
    Thus, there are no edges between $\{w_1,w_2\}$ and $V(F)\sm\{p_2,p_3\}$.
    By symmetry, set $w_1=v_2$ and $w_2=v_0$.
    
    If $v_3=p_2$, then $v_4\neq p_1$, $\lab p_2,p_1,v_4,v_1\rabgb\niso\claw$
    implies $p_1v_4\in E(\bG)$, and since 
    $\lab p_1,v_4,p_2,v_1,p_3,p_4,t_3,t_4\rabgb\niso\Gt$, we have $v_4z\in E(\bG)$
    with $z\in\{p_3,p_4,t_3,t_4\}$, implying $\lab v_4,p_1,v_2,z\rabgb\iso\claw$.
    Hence $v_3\neq p_2$.
    Then $\lab v_2,v_4,p_2,v_0\rabgb\niso\claw$ implies $p_2v_4\in E(\bG)$, 
    $\lab p_2,p_1,v_4,v_1\rabgb\niso\claw$ implies $p_1v_4\in E(\bG)$, and from
    $\lab p_1,v_4,p_2,v_1,p_3\car p_4,t_3,t_4\rabgb\niso\Gt$ we have $v_4z\in E(\bG)$
    with $z\in\{p_3,p_4,t_3,t_4\}$, implying $\lab v_4,p_1,v_2,z\rabgb\iso\claw$,
    a contradiction.
  \end{mylist}
\end{mylist}
~

\vspace*{-15mm}
\end{proof}

\section{Concluding remarks}
\label{sec-concluding}

{\bf 1.} A $\Gt$-closure of a graph $G$, as defined in 
Section~\ref{sec-Gt-closure}, is not unique in general. However, in view 
of Theorem~\ref{thmA-main}, it is unique in 3-connected $\{\claw,\Gt\}$-free graphs
since each such graph is Hamilton-connected, hence has complete closure.

\bs

\noi
{\bf 2.}  
The source codes of our proof-assisting programs are available at \cite{computing}.
The codes are written in SageMath~9.6 and use some of its functions. 
We thank the SageMath community~\cite{Sage} for developing a valuable open-source 
mathematical software.

\section{Acknowledgement}
The research was supported by project GA20-09525S of the Czech Science 
Foundation.

\end{document}